\documentclass[reqno]{amsart}

\usepackage{amssymb}
\usepackage{amsmath}
\usepackage{mathtools}
\usepackage{graphicx}
\usepackage{bbm}
\usepackage{array}
\usepackage{eepic}
\usepackage{multicol}
\usepackage[usenames,dvipsnames]{xcolor}
\usepackage{color}
\usepackage[linktocpage=true]{hyperref}
\usepackage[hypcap=true]{caption}
\usepackage[cmtip,all]{xy}
\usepackage{enumitem}
\usepackage{leftidx}
\usepackage[dvipsnames]{xcolor}
\usepackage[mathscr]{euscript}
\usepackage{setspace}

\usepackage{parskip}

\makeatletter
\def\thm@space@setup{%
  \thm@preskip=\parskip \thm@postskip=0pt
}
\makeatother

\allowdisplaybreaks

\numberwithin{equation}{section}

\theoremstyle{plain}
\newtheorem{theorem}{Theorem}[section]
\newtheorem{lemma}[theorem]{Lemma}
\newtheorem{proposition}[theorem]{Proposition}
\newtheorem{corollary}[theorem]{Corollary}

\theoremstyle{definition}
\newtheorem{definition}{Definition}[section]

\newtheorem{conjecture}[definition]{Conjecture}

\newtheorem{remark}[definition]{Remark}

\renewenvironment{proof}{{\bf \noindent Proof.}}{\qed}

\addtocounter{tocdepth}{0}

\newcommand{\CC}{\mathbb{C}}

\newcommand{\EE}{\mathbb{E}}
\newcommand{\FF}{\mathbb{F}}

\newcommand{\PP}{\mathbb{P}}

\newcommand{\RR}{\mathbb{R}}

\newcommand{\U}{\mathcal{U}}

\newcommand{\Ee}{\mathscr{E}}

\newcommand{\into}{\hookrightarrow}

\DeclareMathOperator*\esssup{\mathrm{ess \, sup}}
\DeclareMathOperator*\essinf{\mathrm{ess \, inf}}
\DeclareMathOperator*\supp{\mathrm{supp}}
\DeclareMathOperator*\hsupp{\mathrm{-supp}}

\def\1{\mathbf{1}}
\def\Id{\mathrm{Id}}
\def\H{{H}}

\def\M{\mathcal{M}}
\def\A{\mathcal{A}}
\def\N{\mathcal{N}}
\def\Zent{\mathcal{Z}}
\def\Ad{\mathrm{Ad}}
\def\Homeo{\mathrm{Homeo}}

\def\B{\mathcal{B}}

\def\Ball{\mathrm{Ball}}
\def\CB{\mathcal{C B}}

\def\L{\mathcal{L}}

\def\cb{\mathrm{cb}}
\def\Aut{\mathrm{Aut}}
\def\Tr{\mathrm{Tr}}

\def\Iso{\mathrm{Iso}}

\def\R{\mathcal{R}}

\def\mintensor{\otimes_{\min}}

\def\weaktensor{ \, \overline{\otimes} \, }

\def\osprojtensor{\, \widehat{\otimes} \,}

\def\BM{\mathfrak{B}}

\begin{document}

\title[Crossed-product extensions]{
  Crossed-products extensions\\
  of $L_p$-Bounds for amenable actions
}

\author[Gonz\'alez-P\'erez]{
  Adri\'an M. Gonz\'alez-P\'erez
}
\address{
  KU Leuven
  Celestijnenlaan 200B, 3001 Leuven
}
\email{
  adrian.gonzalezperez@kuleuven.be
}

\thanks{
  The author has been partially supported by the
  ICMAT-Severo Ochoa Excellence Programme
  \texttt{SEV-2015-0554} and
  and by the European Research Council
  consolidator grant \texttt{614195}.
}

\date{}

\begin{abstract}
  We will extend the transference results in \cite{NeuRic2011,CasSall2015}
  from the context of noncommutative
  $L_p$-spaces associated with amenable groups to that of noncommutative
  $L_p$-spaces over crossed products of amenable actions. Namely, if
  $T_m:L_p(\L G) \to L_p(\L G)$ is a completely bounded operator, where
  $\L G \subset \B(L_2 G)$ is the von Neumann algebra of $G$, then,
  we will see that
  $\Id \rtimes T_m: L_p(\M \rtimes_\theta G) \to L_p(\M \rtimes_\theta G)$
  is also completely bounded and that
  \[
    \| \Id \rtimes T_m: L_p(\M \rtimes_\theta G) \to L_p(\M \rtimes_\theta G) \|_\cb
    \leq \| T_m \|_\cb
  \]
  provided that $\theta$ is amenable and trace-preserving. Furthermore,
  our construction allow to extend $G$-equivariant completely bounded
  operators $S: L_p(\M) \to L_p(\M)$ to the crossed-product, so that
  \[
    \|S \rtimes \Id: L_p(\M \rtimes_\theta G) \to L_p(\M \rtimes_\theta G) \|_\cb
    \leq 
    C^\frac{1}{p} \, \| S \|_\cb
  \]
  whenever $\theta$ is trace-preserving, amenable and its generalized
  F{\o}lner sets satisfy certain accretivity property measured by the
  constant $1 \leq C$. As a corollary we obtain stability results for
  maximal $L_p$-bounds over crossed products. Such results imply the
  stability under crossed products of the \emph{standard assumptions}
  used in \cite{GonJunPar2015} to prove a noncommutative generalization
  of the spectral H{\"o}mander-Mikhlin theorem. 
\end{abstract}

\maketitle

\section*{\bf Introduction}
\label{CP}

The purpose of this article is to study transference results for operators
acting on the $L_p$-spaces of crossed products. In order to state and
prove our results we will need to recall briefly in this introduction
some definitions concerning noncommutative $L_p$-spaces, completely
bounded operators, crossed products and non-commutative maximal
inequalities. We will also provide suitable references for the material
here summarized and formulate the main results of the text.

\subsection*{\it Noncommutative $L_p$-spaces.}
Let $\M \subset \B(\H)$ be a von Neumann algebra. If
$\tau_\M: \M_+ \to [0,\infty]$ is a \emph{normal, semifinite and faithful}
trace, or a n.s.f. trace in short, then we have a very
well understood theory of noncommutative integration. In particular, we can
construct a family of Banach spaces, called the \emph{noncommutative
$L_p$-spaces}. Such spaces are given by completion with respect to the norm
$\| x \|_p = \tau_{\M}(|x|^p)^{1/p}$, see \cite{PiXu2003} for more information.
Naturally, this construction generalizes the classical $L_p$-spaces whenever
$\M$ is abelian and $\tau_\M$ is given by integration against a measure. We will
denote the $L_p$-spaces associated with a trace by $L_p(\M,\tau_\M)$, omitting
the dependency on the trace when it can be understood from the context. If
$\M = \B(\H)$ and $\tau_\M = \Tr$ is the canonical trace the resulting
spaces are called the Schatten classes and denoted by $S_p(\H)$ or $S_p$ if
the dependency on the Hilbert space can be understood from the context. As
it is customary, when $\H = \ell_2^n$ we will denote $S_p(\H)$ by $S_p^n$.

\subsection*{\it Completely bounded maps.}
Throughout this article we will use liberally the language of operator spaces.
Recall that the category of operator spaces can be defined as that of closed
linear subspaces $E \subset \B(\H)$ with morphisms given the, so called,
\emph{completely bounded} maps $\phi: E \to F$. I.e. linear maps such that
their matrix amplifications $\Id \otimes \phi: M_n[E] \to M_n[F]$ are
uniformly bounded in $1 \leq n$. The spaces $M_n[E], M_n[F] \subset 
\B(\ell_2^n \otimes_2 H)$ are just the spaces of $E$-valued,
resp. $F$-valued, matrices with the norm inherited from
$M_n[\B(\H)] = \B(\ell_2^n \otimes_2 H)$. We define the
completely bounded norm as
\[
  \| \phi \|_\cb =
  \| \phi: E \to F \|_\cb =
  \sup_{1 \leq n} \Big\{ \big\| \Id \otimes \phi: M_n[E] \to M_n[F] \big\| \Big\}.
\] 
The space of all completely bounded maps will be denoted by $\CB(E,F)$ and
the term completely bounded will often be shorten to \emph{c.b}. Similarly
a map $\phi$ is completely positive iff all of its matrix amplification
are positivity preserving maps. The category
of operator spaces is closed under quotients, subsets, interpolation and other
operations, see \cite{Pi2003, EffRu2000Book} for more information.
We shall also point out that $L_p(\M)$ can be endowed with a canonical operator
space structure compatible with interpolation. Such structure is given by
interpolation between $L_\infty(\M) = \M \subset \B(\H)$ and $L_1(\M)$, which
will be identified with the predual of the \emph{opposite algebra}, i.e.
the algebra $\M$ with multiplication reversed, see
\cite[Chapter 7, p. 138-141]{Pi2003} for the details.

\subsection*{\it The canonical trace.}
Let $G$ be a locally compact and
Hausdorff group and let $\L G \subset \B(L_2 G)$ be the
von Neumann algebra generated by the left regular representation $\lambda$.
Such von Neumann algebra is commonly understood as a noncommutative
generalization of the $L_\infty$-space over the abelian Pontryagin dual $\hat{G}$.
The algebra $\L G$ carries a natural normal, semifinite and faithful
(\emph{n.s.f.} in short) weight $\tau_G: \L G \to [0,\infty]$, given
by extension of
\[
  \tau_G(\lambda(f)) = \tau_G \Big( \int_G f(g) \, \lambda_g \, d\mu(g) \Big) = f(e),
\]
where $f \in C_c(G) \ast C_c(G)$, see \cite[Chapter 7]{Ped1979} for the details
of the construction of such functional. Such weight coincides with integration
against the Haar measure over the dual group if $G$ is abelian. Observe also that
$\tau_G$ is a tracial weight whenever $G$ is unimodular. Here, we will work in
the context of unimodular groups and the $L_p$-spaces associated with the
canonical trace will be denoted by $L_p(\L G)$. Such weight is sometimes
called the Plancherel weight since the map $f \mapsto \lambda(f)$ becomes a
unitary isometry $\lambda: L_2(G) \to L_2(\L G)$.

\subsection*{\it Crossed Products.}
Let $\M \subset \B(\H)$ be a von Neumann algebra,
$\tau_{\M}: \M_+ \to [0,\infty]$ a n.s.f.
trace and $\theta:G \to \Aut(\M)$ a normal and trace preserving action of
$G$. We define the (spatial) \emph{crossed product} 
$\M \rtimes_\theta G \subset \B(\H \otimes_2 L_2 G)$ as the von Neumann
algebra generated by the images of $g \mapsto \1 \otimes \lambda_g$ and 
$\iota: \M  \into \B(\H \otimes_2 L_2 G)$ given by
\[
  (\iota x)(\xi)(g) = \theta_{g^{-1}}(x) \xi(g),
\]
where $\xi \in L_2(G;\H) \cong \H \otimes_2 L_2(G)$. Observe that
$\M \rtimes_\theta G$ is spanned, after taking weak-$\ast$ completions,
by binomials of the form $\iota x \cdot \lambda_g$. We will usually
denote such binomials by $x \rtimes \lambda_g$.
By \cite{Haa1979, Haa1979ii} we know that there is a faithful and equivariant
operator valued weight $\EE_{\M}: (\M \rtimes_\theta G)_+ \to \M_+^\wedge$
generalizing the Plancherel weight when $\M = \CC$. After composing with
$\tau_\M$ we obtain a n.s.f. weight $\tau = \tau_\M \circ \EE_\M$
that generalizes both $\tau_G$ over
$\1 \otimes \L G \subset \M \rtimes_\theta G$ and $\tau_\M$ over
$\M \rtimes \1 \subset \M \rtimes_\theta G$. It is easy to see that, since
$\theta$ is trace-preserving, $\tau$ is a tracial weight whenever $G$ is
unimodular.

\subsection*{\it Multipliers.}
Our aim in this paper is to transfer the complete boundedness of certain
operators acting on $L_p(\M)$ and $L_p(\L G)$ to $L_p(\M \rtimes_\theta G)$.
Let $m \in L_\infty(G)$. We will denote by $T_m: L_2(\L G) \to L_2(\L G)$
the so-called \emph{Fourier multiplier} of \emph{symbol} $m$,
i.e. the operator given by linear extension of
\[
  T_m \Big( \int_G f(g) \lambda_g \, d\mu(g) \Big)
  = \int_G m(g) f(g) \lambda_g \, d\mu(g).
\]
It is a trivial consequence of the
Plancherel theorem that such operator is bounded in $L_2(\L G)$. The
boundedness in $L_p(\L G)$ is a subtle question that has received much attention
in the past, both in the commutative setting and in the noncommutative one,
see for instance \cite{Pi1995, Har1999Fourier, JunMeiPar2014,
JunMeiPar2014NoncommRiesz} and references within. Fourier multipliers have a
close relative in the so called \emph{Schur multiplers}. Let
$a = [a_{i, j}]_{i,j}, b = [b_{i, j}]_{i, j} \subset \B(\H)$ be matrices.
Their \emph{Schur product} is defined as the cellwise product
$a \bullet b = [a_{i, j} \, b_{i, j}]_{i, j}$.
Any operator given by multiplication with respect
to a fixed $a$ is called a Schur multiplier.
Observe that, since the matrix units $[e_{i,j}]_{i,j}$ form an orthogonal
basis for $S_2(\H)$, if $a \in \ell_\infty \otimes \ell_\infty$, then, the
operator $b \mapsto a \bullet b$ is bounded. Again, determining when a
Schur multiplier is bounded in $S_p(\H)$, for $p \neq 2$, is a
difficult problem. When $\H = \ell_2(\Gamma)$,
for a discrete group $\Gamma$, we can define the \emph{Herz-Schur multipliers}
associated with a \emph{symbol} $m \in \ell_\infty(\Gamma)$ by
\[
  M_m(e_{g,h}) = m(g \, h^{-1}) \, e_{g,h}.
\]
Herz-Schur and Fourier multipliers are
extremely close object since the former restrict to the later
when we restrict the range from $\B(\ell_2 \Gamma)$ to the subalgebra
$\L \Gamma$, see the beginning of Section \ref{CP.S2} for more information.
Recall also that Herz-Schur multipliers can be defined for general locally
compact groups just by extension of the pointwise multiplication by the
integral kernels $k \in C_c(G \times G)$ associated to bounded operators,
the more or less straightforward details can be consulted at the beginning
of \cite{CasSall2015} or \cite{LaffSall2011}.

\subsection*{\it Maximal Inequalities.}
When $\M$ is a hyperfinite von Neumann algebra and $E$ is an
operator space, there is a notion, due to Pisier, of $E$-valued
noncommutative $L_p$-spaces, see \cite{Pi1998}. Indeed, if $\M = M_n(\CC)$
is a matrix algebra, then $S^n_p[E]$ can be defined by operator space
interpolation as follows
\[
  S^n_p[E] := \big[S^n_1 \osprojtensor E, M_n(\CC) \mintensor E\big]_{\frac{1}{p}},
\]
where $\osprojtensor$ and $\mintensor$ are the operator space projective
tensor product and minimal tensor product respectively.
By hyperfiniteness we can approximate $L_p(\M;E)$ by taking direct sums
and unions of such finite dimensional $E$-valued Schatten classes. This
construction was further generalized to the context of QWEP von Neumann
algebras by Junge in an unpublished work. A particularly important case of the
above construction is that of $E = \ell_\infty$. The reason being
that the $L_p(\ell_\infty)$-norm can be used to define maximal
operators in the noncommutative setting. Indeed, recall that if
$(x_n)_n \subset \M$ is a family of noncommuting operators their
supremum is an ill-defined object. Nevertheless, thank to the
$\ell_\infty$-valued noncommutative $L_p$-spaces we can speak
unambiguously about the $L_p$-norm of the supremum just by taking
\[
  \Big\| {\sup_{1 \leq n}}^{+} \big\{ x_n \big\} \Big\|_{L_p(\M)} := \big\| (x_n)_n \big\|_{L_p(\M;\ell_\infty)},
\]
where the ${\sup}^+$ in the left hand side is a purely symbolical element.
It is also worth recalling that $L_p(\M;\ell_\infty)$ spaces can be
characterized as the elements inside $\ell_\infty[L_p(\M)]$ admitting a
factorization of the form $x_n = \alpha \, y_n \, \beta$, where
$\alpha, \beta \in L_{2p}(\M)$ and
$(y_n) \in \ell_\infty \weaktensor \ell_\infty$. In fact the norm is given
by
\[
  \big\| (x_n)_n \big\|_{L_p(\M;\ell_\infty)} = \inf \Big\{ \| \alpha \| \, \sup_{n} \| y_n \| \, \| \beta \| \Big\},
\]
where the infimum is taken over all possible decompositions
$x = \alpha \, y \, \beta$. When $(x_n)_n$ is a positive
element the quantity above can be reduced to
\begin{equation}
  \label{CP.S0.max}
  \big\| (x_n)_n \big\|_{L_p(\M;\ell_\infty)} = \inf_{x_n \leq z} \big\{ \| z \|_{L_p(\M)} \}.
\end{equation}
Indeed the above decomposition allows to generalize the $L_p(\M;\ell_\infty)$-spaces
to the context of non hyperfinite von Neumann algebras $\M$. The noncommutative
$L_p(\ell_\infty)$-spaces have been used in the past to generalize certain
maximal inequalities to the von Neumann algebra setting. Concretely,
such technique has been employed in the past to prove noncommutative
versions of the Doob maximal inequality for martingales \cite{Jun2002Doob} and
noncommutative ergodic theorems \cite{JunXu2007}. In \cite{GonJunPar2015},
maximals bounds were used to prove a principle of boundedness of Fourier
multipliers by maximal operators in the noncommutative setting.

\subsection*{\it Summary.}
After recalling a few facts from amenable actions in Section \ref{CP.S1} we
will state and prove our main results in Section \ref{CP.S2}. In particular,
we are going to see that if $T_m:L_p(\L G) \to L_p(\L G)$
is a completely bounded Fourier multiplier, then its crossed product
extension $\Id \rtimes T_m$, given by
$(\Id \rtimes T_m)(x \rtimes \lambda_g) = x \rtimes m(g) \lambda_g$ is
completely bounded in $L_p(\M \rtimes_\theta G)$, provided that the action
$\theta$ is \emph{amenable}, see \cite[Section 4.3]{BroO2008}
or \cite[Chapter 4]{Zimmer1984} for a precise definition. Furthermore, our
technique yields that
\begin{equation}
  \label{CP.eq.0}
  \| \Id \rtimes T_m \|_{\CB(L_p(\M \rtimes_\theta G))}
  \leq \| S_m \|_{\CB(S_p)}
  \leq \| T_m \|_{\CB(L_p(\L G))},
\end{equation}
see Corollary \ref{CP.cor.bound}.
The techniques involved in the proof of such result are a generalization of the
theorems in \cite{NeuRic2011} and \cite{CasSall2015} from amenable groups to
amenable actions. One of the novelties of our approach is that it allows us to
transfer, not just Fourier multipliers acting on the $G$-component of
$\M \rtimes_\theta G$, but $\theta$-equivariant operators acting of $\M$.
Indeed, strengthening the amenability of $\theta$ by imposing an
accretivity condition on its generalized ``F{\o}lner sets'' gives a
transference result for any completely bounded and $\theta$-equivariant
operator $S$ in $\CB(L_p(\M))$ as follows
\begin{equation}
  \label{CP.eq.1}
  \| S \rtimes \Id \|_{\CB(L_p(\M \rtimes_\theta G))}
  \leq C^\frac{1}{p} \, \| S \|_{\CB(L_p(\M))},
\end{equation}
where $C \geq 1$ is a constant measuring the accretivity of such sets, see
Corollary \ref{CP.cor.bound}. In every example of amenable actions we have worked
so far we can build approximating sequences whose accretivity constant is $C=1$.
We conjecture that such is the case for all amenable actions.
In Section \ref{CP.S1} we will state precisely the amenability condition required
for our theorems and review briefly the equivalent definitions of amenability
for actions.

In Section \ref{CP.S3} we will prove an operator-valued extension of the
transference results described above. Our extension of the transference
results to the $\ell_\infty$-valued case
allows us to obtain maximal \emph{strong-type maximal inequalities} in
crossed products. Concretely, if $(T_n)_{n \geq 0}$ is a family of
completely positive Fourier multipliers over
$L_p(\L G)$ and $(S_m)_{m \geq 0}$ is a family of completely positive
and $\theta$-equivariant operators in $L_p(\M)$ and $\| u \|_p \leq 1$
we have an inequality of the form
\begin{eqnarray}
  &                  & \Big\| \, {\sup_m}^{+} {\sup_n}^{+} \big\{ (S_m \rtimes T_n)(u) \big\} \Big\|_{L_p(\M \rtimes_\theta G)} \label{CP.eq.2}\\
  & \lesssim^{(\cb)} & C^\frac{1}{p} \, \big\| (T_m)_m: L_p(\L G) \to L_p(\L G; \ell_\infty) \big\|_\cb \, \big\| (S_n)_n: L_p(\M) \to L_p(\M;\ell_\infty) \big\|_\cb \nonumber
\end{eqnarray}
whenever $\theta$ has an approximating sequence with accretivity constant $C$.
Observe that such inequality is a trivial consequence of Fubini type
argument when $\M$ and $G$ are abelian and the action $\theta$ is trivial,
since $\M \rtimes_\theta G = \M \weaktensor \L G$.

As a consequence of those maximal inequalities we obtain that the completely
bounded Hardy-Littlewood inequalities, denoted by $\mathrm{CBHL}_p$, in 
\cite{GonJunPar2015} are stable under crossed products if natural
invariance conditions are satisfied. Since the rest of the so-called
\emph{standard assumptions} in \cite[Definition 2.6]{GonJunPar2015}
are easily verified for crossed products, we obtain that
the standard assumptions are stable under crossed products, see Theorem
\ref{CP.thm.Stability}. Observe that, for that application, we could have
just used the amenability of $G$ since, by
\cite[Remark 2.4/2.5]{GonJunPar2015}, the standard assumptions imply
amenability.

\section{\bf Amenability of actions}
\label{CP.S1}

The purpose of this section is to recall a few facts from the theory of
amenable actions and provide suitable references. Up to Definition
\ref{CP.def.Folner} all the material here presented is 
standard and we include it for the sake of completeness. 

Let $(X, \Sigma_X, \nu)$, or simply $(X, \nu)$ if the
$\sigma$-algebra is understood from the context, be a $\sigma$-finite
measure space. We will say that a group homomorphism
$\theta: G \to \Aut(X,\nu)$ is
action of $G$ on $X$ iff the map $(g,x) \mapsto \theta_g(x) = g \, x$ is
measurable and $(\theta_g)_\ast \nu$ and $\nu$ are mutually absolutely
continuous. When $(\theta_g)_\ast \nu(E) = \nu(\theta_g E) = \nu(E)$ we
will say that $\theta$ is $\nu$-preserving. Throughout this text we will
assume that every measure space $(X,\Sigma_X,\nu)$ is given the Borel
structure of an underlying locally compact topological space and that
the measure $\nu$ is regular. Recall also that the action $\theta$ extends
trivially to an action over the functions on $X$. We are going to denote it,
perhaps ambiguously, by $\theta_g f(x) = f(\theta_{g^{-1}} x)$. As usual,
if there is no confusion we may just write $f(g^{-1} \, x)$ or $g \, f$
instead of $f(\theta_{g^{-1}} x)$.

A group $G$ is said to be amenable if there is a translation
invariant mean $m \in L_\infty(G)^\ast$, i.e. an element
$m \in L_\infty(G)^\ast$ such that $m(f) \geq 0$ for every $f \geq 0$,
$m(\chi_{G})=1$ and $m(f(g^{-1} \cdot)) = m(f)$ for every $g \in G$. We
now define a weaker notion of amenability for an action on a von Neumann
algebra.

\begin{definition}
  \label{CP.def.AmenAct}
  Let $\theta: G \to \Aut(X,\nu)$ be an action, we will say that it is
  amenable iff there is a (not necessarily normal) $\theta$-equivariant
  conditional expectation $\Ee:L_\infty(X) \weaktensor L_\infty(G) \to
  L_\infty(X)$, i.e. a unital, positivity preserving and $L_\infty(X)$-linear
  map, such that 
  \[
    \Ee \, f(\theta_{g^{-1}} \cdot, g^{-1} \, \cdot)  = \theta_g \, \Ee(f).
  \]
  
  If $(\M, \varphi)$ is a von Nuemann algebra and $\varphi$ a normal,
  semifinite and faithful weight, an action $\theta: G \to \Aut(\M)$
  is amenable iff its restriction to the abelian subalgebra
  $(\Zent(\M), \varphi{|}_{\Zent(\M)})$ is amenable, where $\Zent(\M)$
  denotes the center of $\M$.
\end{definition}

Observe that, trivially, if $G$ is amenable all of its actions are
amenable, just take $\Ee = \Id_{L_\infty(X)} \otimes m$, for any
$G$-invariant mean $m$. Reciprocally if $G$ acts amenably on a one-point
space then $G$ is amenable. The flexibility gained is that non-amenable
groups may have nontrivial amenable actions. We may also recall that if
$G$ acts amenably in a probability space $(X,\nu)$ and the measure $\nu$
is invariant, then, the composition $\nu \circ \Ee$ is an invariant mean.
The same holds for finite measure spaces with a $\theta$-invariant measure.

Like in the case of amenability there are several equivalent characterizations
of the property. The definition we have introduced above is not the one that appeared
first in the literature. The oldest one, to the knowledge of the author, is
that an action $\theta:G \to \Aut(X,\nu)$ is amenable iff every affine action
on a weak-$\ast$ compact convex set \emph{subordinated to $\theta$} has a
fixed point. A weak-$\ast$ compact convex $G$-set $K \subset E^*$ is said to be
subbordinated to $\theta: G \to \Aut(X,\nu)$ iff $E^*$ can be constructed by
tensoring $L_\infty(X)$ with some dual space $E_0^*$ and twisting with a
$1$-cocycle $\alpha: G \to \BM(X, \Iso(E_0))$, where $\BM(X,A)$ is the space of
Borelian functions. A very detailed introduction to such concept can be found in
\cite[Chapter 4]{Zimmer1984}. Of course, when $X = \{ p \}$, we get that any
affine action of $G$ in a compact weak-$\ast$ closed subset has a
fixed point, a condition long known to be equivalent to amenability, see
\cite{Pa1988Amenability}. Amenable actions were introduced in the pioneering
works of Zimmer, see {\cite{Zimmer1977, Zimmer1978, Zimmer1978Pairs,
Zimmer1978Induced}} following earlier results of Furstenberg
\cite{Furs1973Boundary}. The equivalence with the
definition given here was proved in \cite{AdamEllGior1994}. 
We shall also point out that the notion of amenability stated here is
sometimes referred to as Zimmer-amenability. It shall not be confused with
the very different notion of $X$ admitting a $\theta$-invariant mean
$m \in L_\infty(X)^\ast$ which is sometimes also referred to as amenability
for an action.

It is important to recall that amenability of actions can be defined for
continuous actions on topological spaces. Pretty much in the same way in
which measurable groups are somehow the same objects as topological groups,
see \cite[Chapter 5:6]{Va1985QuantumBook}, topological amenable actions
are {\it the same object} as Borel amenable actions. In order to clarify this we
will need the following proposition. Recall that we are going to denote by
$\PP(G)$, the probability measures with the $\sigma(C_0(G))$-topology and by
$\PP_0(G)$ the subset of all absolutely continuous ones with respect to the
Haar measure.

\begin{proposition}
  \label{CP.prp.BorelMaps}
  Let $\theta:G \to \Aut(X,\nu)$ be an action. It is amenable iff for every
  $m \in \PP_0(G)$, $\epsilon > 0$ and $K \subset G$ a compact subset there is
  a Borel map $\mu: X \to \PP_0(G)$ such that
  \begin{equation}
    \label{CP.eq.BorelMaps.1}
    \sup_{g \in K} \int_X \| g \, \mu^{x} - \mu^{g \, x} \|_1  \, d \, m (x) < \epsilon,
  \end{equation}
  where $g \, d \, \mu^{x}(h) = d \, \mu^{x}(g^{-1} \, h)$. 
\end{proposition}

Whenever a net $(\mu_\alpha)_\alpha$ satisfies condition
\eqref{CP.eq.BorelMaps.1} for every $m$, $\epsilon$ and $K$ provided that
$\alpha$ is large we will say that $(\mu_\alpha)_{\alpha}$ is
\emph{asymptotically equivariant}. Observe that the condition in the
proposition above is equivalent to the existence of an asymptotically
equivariant net. To see that, just denote by $\mu_{m,\epsilon, K}$ the Borel
measurable map in Proposition \ref{CP.prp.BorelMaps} and by $A$ the net
given by all triples $(m,\epsilon,K)$ with the natural order.

\begin{proof}
  Given any Borel map $\mu: X \to \PP_0(G)$ we can associate to it a unital, positivity
  preserving and $L_\infty(X)$-linear map
  $\Ee_\mu: L_\infty(X) \weaktensor L_\infty(G) \to L_\infty(X)$ given by
  \[
    \Ee_\mu(f)(x) = \int_G f(x,g) \, d \, \mu^{x}(g).
  \]
  Clearly all such maps have norm bounded by $\| \Ee_\mu(\1) \|_\infty = 1$.
  The space of bounded maps $\B(L_\infty(X \times G), L_\infty(X))$ is a dual
  Banach space since
  \[
    \begin{array}{rclll}
      \B(L_\infty(X \times G), L_\infty(X)) & = & L_\infty(X) \weaktensor L_\infty(X \times G)^\ast &   &\\
                                            & = & L_1(X)^\ast \weaktensor L_\infty(X \times G)^\ast &   &\\
                                            & = & (L_1(X) \osprojtensor L_\infty(X \times G))^\ast  & = & L_1(X; L_\infty(X \times G))^\ast,
    \end{array}
  \]
  and the pairing is given by extension of
  \[
    \langle m \otimes f ,\Ee \rangle = \langle m, \Ee(f) \rangle = \int_X \Ee(f) \, d \, m.
  \]
  Therefore, by the Banach-Alaoglu compactness theorem, the net
  $(\Ee_{\mu_\alpha})_\alpha$ has a weak-$\ast$ accumulation point $\Ee$.
  Since the subset of all conditional expectations is clearly weak-$\ast$
  closed, $\Ee$ is also a conditional expectation. We have to see that if
  $(\mu_\alpha)_\alpha$ is asymptotically invariant, then $\Ee$ is
  equivariant. But that is obvious since we have 
  \begin{eqnarray*}
     \big\langle \Ee_{\mu_\alpha}(\theta_g f) - \theta_g(\Ee_{\mu_\alpha}(f), m \big\rangle & =    & \int_X \int_G f(g^{-1}  \, x, g^{-1} \, h) \, \big\{ d \, \mu_\alpha^{x}(g \, h) - d \, \mu_\alpha^{g^{-1}}(h) \big\} \, d\, m(x)\\
                                                                                            & \leq & \| f \|_\infty \int_X \, \big\| g^{-1} \, \mu_\alpha^{x} - \mu_\alpha^{g^{-1} \, x} \big\|_1  \, d \, m (x)
  \end{eqnarray*}
  and for every $g \in G$ such quantity can be made arbitrarily small.
  
  For the reciprocal we have to use that the space of
  normal conditional expectations from $L_\infty(X \times G)$
  to $L_\infty(X)$ is dense inside the set of all conditional expectations
  with respect to the weak-$\ast$ topology. Notice that, by applying the
  Hahn-Banach theorem in every fibre, normal conditional
  expectations are in correspondence with measurable maps
  $\mu: X \to \PP_0(G)$. If $\Ee$ is an
  equivariant conditional expectation we have that there is a net 
  $(\mu_\alpha)$ of Borel maps with $\Ee_{\mu_\alpha} \to \Ee$ in the
  weak-$\ast$ topology. The net $\mu_\alpha$ is asymptotically
  equivariant. After identifying Borel maps $X \to \PP_0(G)$ with a subset
  of $L_\infty(X;L_1 G )$ we have that the weak-$\ast$ topology of
  $\B(L_\infty(X \times G),L_\infty(X))$ corresponds to the
  $\sigma(L_1(X;L_\infty G))$ topology. In particular, for every $g \in G$
  we have that $g \, \mu_\alpha^{x} - \mu_\alpha^{g \, x}$ tends to zero in
  the $\sigma(L_1(X;L_\infty G))$ topology. In particular $0$ is in the
  $\sigma(L_1(X;L_\infty G))$-closed convex hull of the set of all the maps
  \[
    S_g = \big\{ x \mapsto ( g \, \mu^{x} - \mu^{g \, x} ) \big\}.
  \]
  It is easily seen that such convex set equals the closure of $S_g$ in
  the coarsest linear topology making all maps 
  \[
    \mu \mapsto \int_X \| \mu^{x} \|_1 \, d \, m(x)
  \]
  continuous. Taking a sequence in the convex hull of $S_g$ converging to $0$
  in such topology gives the claim.
\end{proof}

We recall now the definition of amenability for topological actions. We will say
that an action of $G$ by homeomorphisms $\theta: G \to \Homeo(X)$ is a topological
action iff the map $(g,x) \mapsto \theta_g(x)$ is continuous.

\begin{definition}
  \label{CP.def.TopAmen}
  Let $X$ be a locally compact topological space and $\theta$ a topological
  action. The action is said to be amenable iff there is a net of continuous maps
  $\mu_\alpha: X \to \PP(G)$, such that for every $g \in G$
  \[
    \lim_\alpha \, \sup_{x \in X} \| g \, \mu_\alpha^{x} - \mu_\alpha^{g \, x} \|_1 = 0.
  \]
  Similarly, an action of $G$ in a $C^\ast$-algebra $\A$ is amenable iff its
  restriction to the center $\Zent(\A)$ is (topologically) amenable.
\end{definition}

Observe that, since $\PP(G)$ is a compact, each $\mu_\alpha$ can be lifted
to a continuous function on $\beta X$, its Stone-{\v{C}}ech compactification
and so we obtain that, by construction, a continuous action $\theta$ on $X$
is amenable iff its lift $\beta \theta$ to $\beta X$ is amenable.

Such topological definition of amenability appeared in the form above for the
first time in \cite{HigRoe2000Novikov}. In contemporary literature is more
common to see amenable actions defined in terms of topological spaces. The
topic of topological amenable actions has been researched in connection with
exactness for groups, a notion introduced in \cite{KirchWass1999}, since it
was proved in \cite{Oz2000} that a discrete group is exact
iff it has an amenable action on a compact space. See also \cite{Oz2006} for
more on amenable actions.

Recall that we can identify continuous functions $x \in C_c(G;\M)$ with
elements inside $\M \rtimes_\theta G$ and that the operator valued weight
$\EE_\M: (\M \rtimes_\theta G)_+ \to \M_+^\wedge$ satisfies that
\begin{equation}
  \label{CP.eq.L2Cross}
  \EE_\M \big[ x \, x^\ast \big] = \int_G x(g) \, x(g)^\ast \, d \, \mu(g),
\end{equation}
for any $x, y \in \M \rtimes_\theta G$. When working with $C^\ast$ algebraic
crossed products like bellow there is no ambiguity assuming
$\A \subset \A^{\ast \ast}$ to define $\EE_\A$. The characterization below
is easily seen to be equivalent to amenability.

\begin{lemma}{{\bf (\cite[Definition 4.3.1/Lemma 4.3.7]{BroO2008})}}
  \label{CP.lem.Approx}
  An continuous action $\theta: G \to \Aut(\A)$, where $\A$ is a unital
  $C^\ast$-algebra is amenable iff there is a net 
  $(x_\alpha)_\alpha \subset C_c(G;\Zent(\A))$ of compactly supported
  functions satisfying
  \begin{enumerate}[label= \emph{(\roman*)}, ref={(\roman*)}]
    \item \label{CP.lem.Approx.1} $0 \leq x_\alpha(g).$
    \item \label{CP.lem.Approx.2} $\displaystyle{\int_G |x_\alpha(g)|^2 \, d \, \mu(g)} = \1_\A.$
    \item \label{CP.lem.Approx.3}
      $\lim_\alpha \EE_\A \big\{ ((\1 \rtimes \lambda_g) \, x_\alpha - x_\alpha) \, ((\1 \rtimes \lambda_g) \, x_\alpha - x_\alpha)^\ast \big\} \quad \\
       \quad  = \displaystyle{\lim_\alpha \int_G | \theta_g(x_\alpha(g^{-1} \, h)) - x_\alpha(h) |^2 \, d \, \mu(h)} = 0.$
  \end{enumerate}
  Any such net will be called an \emph{approximating sequence}.
\end{lemma}

The following proposition ensures that if $X$ is the Borel space underlying
a compact space and $\theta$ a continuous action, then $\theta$ is amenable
in the measurable sense iff it is amenable in the topological sense.

\begin{proposition}[{{\bf \cite[Proposition 5.2.1]{BroO2008}}}]
  \label{CP.prp.ConToBorel}
  Let $X$ be a compact Hausdorff space and $\theta$ a continuous action
  of $G$ on $X$. Then $\theta$ is (topologically) amenable iff we can take a
  net of asymptotically equivariant Borel maps $\mu_\alpha: X \to \PP_0(G)$.
\end{proposition}

It rests to see that any measurable action comes from a topological action.

\begin{theorem}[{\bf \cite[Theorem 5.7]{Va1985QuantumBook}}]
  For any measurable action $\theta$ of $G$ in $X$ there is a
  compact Hausdorff topological space $Y$, a continuous action $\theta_0$ of
  $G$ on $Y$ and a $\theta_0$-invariant Borel subset $E \subset Y$ such that
  $X$ and $E$ are isomorphic as $G$-spaces.
\end{theorem}

Sometimes the Borel subset $E$ above can be taken to be closed without loss
of generality. One of such situations is when the action preserves a finite
measure. Let $(X,\nu)$ be a finite measure space and the action $\theta$ of
$G$ be $\nu$-preserving. If $E \subset Y$ is like in the theorem above and
$\iota: X \into E \subset Y$ we have that the finite measure
$\iota_\ast \nu \in M(Y)$ is Borelian and its support is a closed $G$-invariant
subset $\supp[\iota_\ast \nu] \subset E$. Restricting to such support amounts
to removing a null set of $X$, see \cite[Lemma 1.3]{AdamEllGior1994}. Similar
results follow for $\nu$-preserving actions when $\nu$ is an infinite regular
measure changing closed sets by locally closed sets.

As a corollary of the following discussion we obtain that any action
$\theta:G \to \Aut(\M)$ of $G$ on a von Neumann algebra $\M$ is amenable
iff it has an approximating sequence $(x_\alpha)_\alpha
\subset C_c(G;\Zent(\M))$ as in Lemma \ref{CP.lem.Approx}. We introduce
now the refinement of amenability of actions that we are going to use
through the next subsections.

\begin{definition}
  \label{CP.def.Folner}
  Let $(\M, \tau)$ be a semifinite von Neumann algebra and denote
  $(\Zent(\M),\tau{|}_{\Zent(\M)})$ by $(L_\infty(X),\nu)$. We say that
  the action $\theta: G \to \Aut(\M)$ has a $C$-approximating sequence
  iff there is a sequence of sets $F_\alpha \subset X \times G$ such that
  \begin{equation}
    \label{CP.def.acretive}
    1 \leq \esssup_{x} \mu \{ g \in G : (x,g) \in F_\alpha \} \leq C \, \essinf_{x} \mu \{ g \in G : (x,g) \in F_\alpha \} < \infty,
  \end{equation}
  and the elements
  \[
    x_\alpha(x,g) = \frac{\chi_{F_\alpha}(x,g)}{\mu \{ g \in G : (x,g) \in F_\alpha \}^{\frac{1}{2}}}
  \]
  form an approximating sequence satisfying \ref{CP.lem.Approx.3} in
  Lemma \ref{CP.lem.Approx}.
\end{definition}

Many natural amenable actions, for example that of $\FF_r$ in its
hyperbolic boundary, admit a $1$-approximating sequence. For instance, we
can just take $F_m$ the set of pairs $(\omega, \eta) \in
\partial \FF_r \times \FF_r$ such that $\omega$ is an infinite reduced word
and $\eta$ is one of the $m$ initial subwords of length less than $m$. 
The existence of $C$-approximating
sequences is stable under natural operations like tensor product extensions
$\Id \otimes \theta: G \to \Aut(\M \weaktensor \M_2)$, diagonal products
$\theta_1 \times \theta_2: G \to \Aut(\M_1 \weaktensor \M_2)$ or tensor
products $\theta_1 \otimes \theta_2: G_1 \times G_2 \to
\Aut(\M_1 \weaktensor \M_2)$. We conjecture that every amenable action 
admits $C$-approximating sequences with $C = 1$.

\section{\bf An asymptotic embedding}
\label{CP.S2}

In this section we are going to prove the main result of this article.
Observe that if $\theta: G \to \Aut(\M)$ is an action and
$\M \rtimes_\theta G$ is the reduced or spatial crossed product, then,
the embedding of $\M \rtimes_\theta G$ into $\B(\H \otimes_2 L_2 G)$ factors
through the subalgebra $\M \weaktensor \B(L_2 G)$. Indeed, after identifying
kernels $k$ in $L_\infty(G \times G;\M)$ with operators in
$\M \weaktensor \B(L_2 G)$, the embedding
$j: \M \rtimes_\theta G \to \M \weaktensor \B(L_2 G)$ is given
by extension of the map sending $u \in C_c(G;\M)$ to the operator with kernel
\[
  k(g,h) = [\theta_g^{-1}(u(g \, h^{-1}))]_{g,h \in G}.
\]
Let $T_m: \L G \to \L G$ be a normal and c.b. Fourier multiplier of symbol $m$
and denote by $(\Id \rtimes T_m)$ its crossed product amplification, i.e. the
normal operator given by linear extension of the map
$x \rtimes \lambda_g \mapsto m(g) \, x \rtimes \lambda_g$. A trivial
calculation show that the isometry $j$ intertwines $\Id \rtimes T_m$
and $\Id \otimes M_m$ as shown below
\[
  \xymatrix{
    \M \rtimes_\theta G \, \ar@{^{(}->}[rr]^-{j} \ar[d]^-{\Id \rtimes T_m} & & \M \weaktensor \B(L_2 G) \ar[d]^-{\Id \otimes M_m}\\
    \M \rtimes_\theta G \, \ar@{^{(}->}[rr]^-{j}                           & & \M \weaktensor \B(L_2 G),   
  }
\]
where $M_m: \B(L_2 G) \to \B(L_2 G)$ is the c.b. Herz-Schur multiplier given
by
\[
  M_m([a_{g \, h}]_{g,h}) = [m(g \, h^{-1}) \, a_{g \, h}]_{g \, h}.
\]
Similarly, let $S: \M \to \M$ be an operator and let us denote by
$S \rtimes \Id$ its crossed product amplification, i.e. the map given
by extension of $x \rtimes \lambda_g \mapsto S(x) \rtimes \lambda_g$. An
straightforward calculation shows that the embedding $j$ intertwines
$S \rtimes \Id$ and $S_\theta$ as follows
\[
  \xymatrix{
    \M \rtimes_\theta G \, \ar@{^{(}->}[rr]^-{j} \ar[d]^-{S \rtimes \Id} & & \M \weaktensor \B(L_2 G) \ar[d]^-{S_\theta}\\
    \M \rtimes_\theta G \, \ar@{^{(}->}[rr]^-{j}                         & & \M \weaktensor \B(L_2 G),
  }
\]
where the map $S_\theta: \M \weaktensor \B(L_2 G) \to \M \weaktensor \B(L_2 G)$
is given by
\[
  S_\theta([x_{g \, h}]) = [\theta_g^{-1} \, S \, \theta_g (x_{g \, h})]_{g,h \in G}.
\]
Therefore, if $S: \M \to \M$ is a normal c.b. and $\theta$-equivariant
operator we obtain that
\[
  \xymatrix{
    \M \rtimes_\theta G \, \ar@{^{(}->}[rr]^-{j} \ar[d]^-{S \rtimes \Id} & & \M \weaktensor \B(L_2 G) \ar[d]^-{S \otimes \Id}\\
    \M \rtimes_\theta G \, \ar@{^{(}->}[rr]^-{j}                         & & \M \weaktensor \B(L_2 G).
  }
\]
Observe that, a posteriori, such intertwining identities imply that if
$M_m$ is completely bounded so is $\Id \rtimes T_m$ and that if
$S: \M \to \M$ is completely bounded and $\theta$-equivariant so is
$S \rtimes \Id$. It is a well-known result, see \cite{BoFend1991},
\cite{CasSall2015}, that the c.b. norm of the Fourier multiplier $T_m$
bounds the c.b. norm of the Herz-Schur multiplier $M_m$. Summing all up,
we obtain the following inequalities
\[
  \begin{array}{rclll}
    \| \Id \rtimes T_m \|_\cb & \leq & \| \Id \otimes M_m \|_\cb & \leq & \| T_m \|_\cb \\
    \| S \rtimes \Id \|_\cb   & \leq & \| S \otimes \Id \|_\cb   & =    & \| S \|_\cb.
  \end{array}
\]
The purpose of this section is to generalize such results from the
crossed product von Neumann algebra $\M \rtimes_\theta G$ to its
noncommutative $L_p$-spaces. The main difficulty stems from the
fact that the isometry $j$ is not trace preserving. In fact,
it is easy to see that if $G$ is a finite group, we have that
\[
  (\tau_\M \otimes \Tr)(j \, \1) = |G| \, \tau(\1),
\]
where $\tau: (\M \rtimes_\theta G)_+ \to [0,\infty]$ is the trace extending
both $\tau_\M$ and $\tau_G$. Therefore $j$ is unbounded in
$L_1(\M \rtimes_\theta G)$ when $G$ is discrete
and infinite. Similar arguments yield that $j$ is ill-defined in $L_1$ when $G$
is noncompact. A way of circumvent this difficulty is to use amenability to
approximate the map $j$ over compact subsets of $G$. This way of proceeding was
used by E. Ricard and S. Neuwirth in \cite{NeuRic2011}, when $\M = \CC$ and
$G$ is a discrete amenable group, to prove that if a Herz-Schur multiplier is
completely bounded in $S_p(L_2 G)$, then so is the Fourier multiplier with the
same symbol in $L_p(\L G)$. Their result was generalized later by M. Caspers
and M. de la Salle in \cite{CasSall2015} to locally compact and amenable
groups. They also
proved that amenability is necessary for such theorem, at least for
$4 \leq p$ an even integer.
We are going to generalize the transference results from amenable groups to
amenable actions and from the $L_p$-spaces of group algebras $\L G$ to the
$L_p$-spaces of crossed products. The way by which we are going to proceed is
to use amenability to approximate $j$ by a net $j_p^\alpha:
L_p(\M \rtimes_\theta G) \to L_p(\M \weaktensor \B(L_2 G))$ of complete
contractions such that they are ``asymptotically isometric''. Then, we can
obtain a complete isometry by taking an ultraproduct of all such maps,
getting
\[
  (j_\alpha)_\alpha^\U:L_p(\M \rtimes_\theta G) \longrightarrow \prod_{\alpha, \U} L_p(\M \weaktensor \B(L_2 G)).
\]
Recall that the ultraproduct above must be understood in the operator space
sense, see \cite[Appendix]{EffRu2000Book}.

Let us start proving the following lemma.

\begin{lemma}
  \label{CP.lem.Iso}
  Let $(\M,\tau)$, $\theta : G \to \Aut(\M)$ be as above and assume
  that $\theta$ is $\tau$-preserving and amenable and that $G$ is
  unimodular. Let $(x_\alpha)_\alpha \subset C_c(G;\Zent(\M))$ be any
  approximating net for $\theta$ and $X_\alpha \in
  \M \weaktensor \B(L_2 G)$ be
  \[
    (X_\alpha \, \xi) (g) = \theta_g^{-1}(x_\alpha(g)) \, \xi(g),
  \]
  where $\xi \in L_2(G;\H)$ The maps $j_p^\alpha: L_p(\M \rtimes_\theta G) \to
  L_p(\M \weaktensor \B(L_2 G))$ given by
  \[
    j_p^\alpha(v) = X_\alpha^{\frac{1}{p}} \, j(v) \, X_\alpha^{\frac{1}{p}}
  \]
  satisfy that
  \begin{enumerate}[ref=(\roman*),label=\emph{(\roman*)}]
    \item \label{CP.lem.Iso1}
    $\displaystyle{\big\| j_p^\alpha: L_p(\M \rtimes_\theta G) \to L_p(\M \weaktensor \B(L_2 G)) \big\|_\cb} \leq 1$,
    for every $1 \leq p \leq \infty$.
    \item \label{CP.lem.Iso2}
    $\displaystyle{\lim_\alpha \big\langle (j_p^\alpha \, u), (j_{p'}^\alpha \, v) \big\rangle} = \langle u, v \rangle$,
    where $\displaystyle{\frac{1}{p} + \frac{1}{p'}} = 1$,
    for every $1 \leq p < \infty$.
  \end{enumerate}
\end{lemma}

\begin{proof}
  The proof of \ref{CP.lem.Iso1} is trivial when $p = \infty$. Proving
  it for $p = 1$ and applying interpolation yields the desired result. Let
  $u \in L_1(\M \rtimes_\theta G)$. We can decompose $u$ as $x = a \, b^\ast$,
  with $\| u \|_2 = \| v \|_2 = \| x \|^\frac{1}{2}$. We have that
  \[
    j_1^\alpha(u) = X_\alpha \, j(a) \, j(b)^\ast \, X_\alpha = (X_\alpha \, j(u)) \, (X_\alpha \, j(v))^\ast.
  \]
  But, clearly
  \[
    \| j_1^\alpha(x) \|_{L_1(\M \weaktensor \B(L_2 G))}
    \leq
    \| X_\alpha \, j(u) \|_{L_2(\M \weaktensor \B(L_2 G))} \,
    \| X_\alpha \, j(v) \|_{L_2(\M \weaktensor \B(L_2 G))}.
  \]
  It is trivial to notice that, since $\tau$ in $\theta$-invariant
  $L_2(\M \rtimes_\theta G) = L_2(\M) \otimes_2 L_2(G)$ and the isomorphism
  is given by
  \[
    \langle u, v \rangle_{L_2(\M \rtimes_\theta G)} = \int_G \tau_\M(u(g)^\ast \, v(g)) \, d \, \mu(g),
  \]
  after identifying $u$ affiliated with $\M \rtimes_\theta G$ with an
  $\M$-valued function of $G$ in the natural way. On the other hand,
  if we denote by $k(g,h) = \theta_g^{-1}(x_\alpha(g)) \, 
  \theta_g^{-1}(u(g \, h^{-1}))$ the kernel of $X_\alpha \, j(u)$, we have
  that 
  \begin{eqnarray}
    \| X_\alpha \, j(u) \|^2_{L_2(\M \weaktensor \B(L_2 G))} & = & (\tau_\M \otimes \Tr) \bigg\{ \Big[ \int_G k(g,h)^\ast \, k(k,h) \, d \, \mu(h) \Big]_{g,k} \bigg\} \nonumber \\
                                                             & = & \int_G \int_G \tau_\M \big\{ | \theta_g^{-1}(x_\alpha(g)) \, \theta_g^{-1}(u(g \, h^{-1})) |^2 \big\} \, d \, \mu(g) \, d \, \mu(h) \nonumber \\
                                                             & = & \int_G \int_G \tau_\M \big\{ | x_\alpha(g) \, u(h^{-1}) |^2 \big\} \, d \, \mu(g) \, d \, \mu(h) \label{CP.prf.eq1}\\
                                                             & = & \int_G \tau_\M \bigg\{ \Big( \int_G | x_\alpha(g) |^2 \, d \, \mu(g) \Big) \, | u(h^{-1}) |^2 \bigg\} \, d \, \mu(h) \nonumber \\
                                                             & = & \int_G \tau_\M \big\{ | u(h) |^2 \big\} \, d \, \mu(h), \label{CP.prf.eq2}
  \end{eqnarray}
  by using the $\theta$-invariance of $\tau_\M$ in \eqref{CP.prf.eq1} and 
  Condition \ref{CP.lem.Approx.2} on Lemma \ref{CP.lem.Approx} as well as the
  unimodularity of $G$ in \eqref{CP.prf.eq1}. The same follows for $v$ and this
  proves \ref{CP.lem.Iso1}.
  
  In order to prove \ref{CP.lem.Iso2} start by noticing that
  \begin{eqnarray*}
    \langle j_p^\alpha(u), j_p^\alpha(v) \rangle & = & \int_G \int_G \tau_\M \big\{  \theta_g^{-1}(x_\alpha(g)) \, \theta_g^{-1}(u(g \, h^{-1})^\ast \, v(g \, h^{-1})) \, \theta_h^{-1}(x_\alpha(h)) \big\} \, d \, \mu(g) \, d \, \mu(h) \\
                                                 & = & \int_G \int_G \tau_\M \big\{  \theta_{g \, h}^{-1}(x_\alpha(g \, h)) \, \theta_{g \, h}^{-1}( u(g)^\ast \, v(g) ) \,  \theta_h^{-1}(x_\alpha(h)) \big\} \, d \, \mu(g) \, d \, \mu(h) \\
                                                 & = & \int_G \int_G \tau_\M \big\{  x_\alpha(g \, h) \, \theta_{g}(x_\alpha(h)) \, u(g)^\ast \, v(g) \big\} \, d \, \mu(g) \, d \, \mu(h) \\
                                                 & = & \int_G \tau_\M \{ u(g)^\ast \, v(g) \} \, d \, \mu(g) + \int_G \tau_\M \big\{ u(g)^\ast \, v(g) \, A \big\} \, d\mu(g),
  \end{eqnarray*}
  where $A$ is just
  \begin{eqnarray*}
    A & =    & \int_G x_\alpha(g \, h) \, \theta_{g^{-1}}(x_\alpha(h)) \, d \mu(h) - \1_\M\\
      & =    & \int_G x_\alpha(g \, h) \, \theta_{g^{-1}}(x_\alpha(h)) \, d \mu(h) - \int_G |x_\alpha(g)|^2 \, d \, \mu(g) \\
      & =    & \int_G x_\alpha(h) \,  \big( \theta_{g^{-1}}(x_\alpha(g \, h)) - x_\alpha(h) \big) \, d \mu(h)\\
      & \leq & \bigg\| \int_G | x_\alpha(h) |^2 \, d \, \mu(h) \bigg\|^\frac{1}{2}_\M \, \bigg( \int_G | \theta_{g^{-1}}(x_\alpha(g \, h)) - x_\alpha(h) |^2 \, d \mu(h) \bigg)^\frac{1}{2}\\
      & =    & \EE_\M \big[ ((\1 \rtimes \lambda_{g^{-1}}) \, x_\alpha - x_\alpha) \, ((\1 \rtimes \lambda_{g^{-1}}) \, x_\alpha - x_\alpha)^\ast \big]^\frac{1}{2} \longrightarrow 0,
  \end{eqnarray*}
  notice that we have used identity \eqref{CP.eq.L2Cross} in the last step. Using
  Condition \ref{CP.lem.Approx.3} of Lemma \ref{CP.lem.Approx} and
  the Dominated Convergence Theorem gives the desired claim.
\end{proof}

We can now proceed to prove the main theorem of this section.

\begin{theorem}
  \label{CP.thm.ultraProdIntertwining}
  Let $(\M,\tau_\M)$, $G$ and $\theta : G \to \Aut(\M)$ be as above with
  $\theta$ amenable. For any $1 \leq p < \infty$ we have a completely
  positive and completely isometric map 
  \begin{eqnarray*}
    L_p(\M \rtimes_{\theta} G) & \xrightarrow{ \quad j_p \quad } & \prod_{\alpha, \mathcal{U}} L_p(\M \weaktensor \B(L_2 G)).
  \end{eqnarray*}
  The isometry $j_p$ satisfies that if $M_m$ and $T_m$ are the Fourier and
  Herz-Schur multipliers associated to the symbol $m$, then
  \begin{equation*}
    \label{CP.thm.CrossDiagramII}
    \xymatrix{
      L_p(\M \rtimes_\theta G) \ar[d]^-{(\Id \rtimes T_m)} \ar[rr]^-{j_p} & & \displaystyle{ \prod_{\alpha, \mathcal{U}} L_p(\M \weaktensor \B(L_2 G))}  \ar[d]^-{(\Id \otimes M_m)^\mathcal{U}}\\
      L_p(\M \rtimes_\theta G)                             \ar[rr]^-{j_p} & & \displaystyle{ \prod_{\alpha, \mathcal{U}} L_p(\M \weaktensor \B(L_2 G))},
    }
  \end{equation*}
  Furthermore, if $\theta$ has a $C$-approximating sequence and
  $S : L_p(\M) \to L_p(\M)$ is a completely bounded
  and $\theta$-equivariant operator, then
  \begin{equation*}
    \label{CP.thm.CrossDiagramI}
    \xymatrix{
      L_p(\M \rtimes_\theta G) \ar[d]^-{(S \rtimes \Id)} \ar[rr]^-{j_p} & & \displaystyle{ \prod_{\alpha, \mathcal{U}} L_p(\M \weaktensor \B(L_2 G))}  \ar[d]^-{(S_\alpha)_\alpha^\mathcal{U}}\\
      L_p(\M \rtimes_\theta G)                           \ar[rr]^-{j_p} & & \displaystyle{ \prod_{\alpha, \mathcal{U}} L_p(\M \weaktensor \B(L_2 G))},
    }
  \end{equation*}
  where
  \[
    \big\| S_\alpha: L_p(\M \weaktensor \B(L_2 G)) \to L_p(\M \weaktensor \B(L_2 G)) \big\|_\cb \leq C^\frac{1}{p} \,\big\| S : L_p(\M) \to L_p(\M) \big\|_\cb.
  \]
\end{theorem}

\begin{proof}
  Let $j_p^\alpha$ be the maps defined in Lemma \ref{CP.lem.Iso}, we define the
  isometry $j_p$ just by $j_p = (j_p^\alpha)_\alpha^\U$ for some principal
  ultrafilter $\U$. Such map is completely contractive since each $j_p^\alpha$
  is. To prove that it is an isometry notice that, for any von Neumann algebra
  $\N$ we have 
  \begin{equation}
    \label{CP.prf.inclusionUltraPrd}
    \prod_{\alpha, \U} L_p(\N) \subset \bigg( \prod_{\alpha, \U} L_{p'}(\N) \bigg)^\ast,
  \end{equation}
  and the embedding is isometric. Indeed, such identity is a consequence, 
  when $1 < p$, of the fact that the dual of the ultraproduct is larger
  than the ultraproduct of the duals, see \cite[pp. 59-63, (2.8.8)]{Pi2003}.
  For $p = 1$, in addition, we have to use the injectivity of the
  ultraproduct construction, see \cite[pp. 59-63, (2.8.2)]{Pi2003}, and apply
  it to the inclusion $L_1(\N) \subset L_1(\N)^{\ast \ast}$. With identity
  \eqref{CP.prf.inclusionUltraPrd} at hand, we have that
  \[
    \begin{array}{rc>{\displaystyle}l>{\displaystyle}l}
      \| j_p \, x  \|_{\prod_{\alpha, \U} L_p} & =    & \| j_p \, x \|_{\left(\prod_{\alpha, \U} L_{p'}\right)^\ast}                                                        & \\
                                               & =    & \sup_{\| h \|_{p'} \leq 1} \big| \big\langle j_p \, x, h \big\rangle \big|                                          & \\
                                               & \geq & \sup_{\| y \|_{p'} \leq 1} \big| \big\langle j_p \, x, j_{p'} \, y \big\rangle \big|                                & \\
                                               & =    & \sup_{\| y \|_{p'} \leq 1} \, \lim_{\alpha, \U} | \big\langle j^\alpha_p \, x, j^\alpha_{p'} \, y \big\rangle \big| & = \sup_{\| y \|_{p'} \leq 1} \big| \big\langle x, y \big\rangle \big| = \| x \|_{L_p}.
    \end{array}
  \]
  Therefore $j_p$ is an isometry. The fact that it is a complete isometry
  follows by similar means.
  
  The intertwining identity concerning $M_m$ and $T_m$ is trivial since all
  of the contractions $j_p^\alpha$ satisfy that
  \[
    j_p^\alpha \, (\Id \rtimes T_m) = (\Id \otimes M_m) \, j_p^\alpha
  \]
  and so does their ultraproduct $j_p$. The second intertwining relation is more delicate.
  The reason is that, if we want $j_p^\alpha$ to intertwine $S \rtimes \Id$
  and $S \otimes \Id$, we need, a priori, to impose $S$ to be $\M_\alpha$-bimodular,
  where $\M_\alpha$ is the von Neumann algebra given by
  \[
    \M_\alpha = \{ \theta_g^{-1} x_\alpha(g) \}^{''}_{g \in G} \subset \Zent(\M). 
  \]
  But such condition is too restrictive. To overcome such difficulty, we will
  assume that net $(x_\alpha)_\alpha$ comes from a $C$-approximating sequence.
  Then, for any $\alpha$ we can define the
  operator $Y_\alpha \in \M \weaktensor \B(L_2 G)$ given by
  \[
    (Y_\alpha \, \xi)(g) = 
    \begin{cases}
      \displaystyle{ \Big( P_{\alpha,g}^{\perp} + P_{\alpha, g} \, \frac{1}{\theta_{g}^{-1} x_\alpha(g)} \Big) } \, \xi(g) & \mbox{ when } g \in G\hsupp[x_\alpha] \\
      \xi(g)                                                                                                               & \mbox{ otherwise, }
    \end{cases}
  \]
  where $P_{\alpha,g} \in \Zent(\M)$ is the orthogonal projection onto
  the support of $x_\alpha(g)$. Clearly, we have that
  \[
    \| Y_\alpha \|_{\M \weaktensor \B(L_2 G)} \leq \max \big\{ 1, \esssup_{x} \mu \{ g \in G : (x,g) \in F_\alpha\}^{\frac{1}{2}} \big\} < \infty
  \]
  and since $Y_\alpha \, X_\alpha = X_\alpha \, Y_\alpha =
  \1_\M \otimes P_{G\hsupp[x_\alpha]}$ we obtain that
  \begin{eqnarray*}
    j_p^\alpha \, (S \rtimes \Id) & = & \underbrace{\Ad_{X^{1/p}_\alpha} \, (S \otimes \Id) \, \Ad_{Y^{1/p}_\alpha}}_{S_\alpha} \, j_p^\alpha,
  \end{eqnarray*}
  where $\Ad_{S}$ is the operator given by $\Ad_S(T) = S^\ast \, T \, S$. All
  that rest to do is to estimate the c.b. norm of $S_\alpha$. We have
  \begin{equation}
    \| S_\alpha \|_\cb \leq \big\| \Ad_{X^{1/p}_\alpha} \big\|_\cb \, \| S \otimes \Id \|_\cb \, \big\| \Ad_{Y^{1/p}_\alpha} \big\|_\cb.
  \end{equation}
  Therefore, if $\lim_{\alpha, \U} \| \Ad_{X^{1/p}_\alpha} \|_\cb \,
  \| \Ad_{Y^{1/p}_\alpha} \|_\cb < \infty$, then the result follows.
  We have that 
  \[
    \begin{array}{rcll>{\displaystyle}l}
      \label{CP.prf.eq.1}
      \big\| \Ad_{X^{1/p}_\alpha} \big\|_\cb & \leq & \| X_\alpha \|^{\frac{2}{p}}_{\M \weaktensor \B(L_2 G)} & \leq & \Big( \essinf_{x} \mu \{ g \in G : (x,g) \in F_\alpha\} \Big)^{-\frac{1}{p}}\\
      \big\| \Ad_{Y^{1/p}_\alpha} \big\|_\cb & \leq & \| Y_\alpha \|^{\frac{2}{p}}_{\M \weaktensor \B(L_2 G)} & \leq & \max \big\{ 1, \esssup_{x} \mu \{ g \in G : (x,g) \in F_\alpha\} \big\}^{\frac{1}{p}}.
    \end{array}
  \]
  Using property \eqref{CP.def.acretive} in the definition of
  $C$-approximating sequence gives
  \[
    \big\| S_\alpha : L_p(\M \weaktensor \B(L_2 G)) \to L_p(\M \weaktensor \B(L_2 G)) \big\|_\cb
    \leq
    C^\frac{1}{p} \, \big\| S: L_p(\M) \to L_p(\M) \big\|_\cb
  \]
  and that concludes the proof.
\end{proof}

As a corollary we obtain that, for any amenable action, if $M_m$ is a
completely bounded Herz-Schur multiplier in $S_p(L_2 G)$ then
$\Id \rtimes T_m$ is c.b. in $L_p(\M \rtimes_\theta G)$. But
\cite[Theorem 4.2]{CasSall2015} asserts that if $T_m$ is c.b in $L_p(\L G)$,
so is $M_m$ in $S_p(L_2 G)$. Therefore, we get that if $T_m$ is c.b. so is
$\Id \rtimes T_m$. Similarly, if $S$ is a $\theta$-equivariant c.b. operator
over $L_p(\M)$ we have that $S \rtimes \Id$ is also c.b. The corollary bellow
sumarises both statements

\begin{corollary}
  \label{CP.cor.bound}
  Let $\theta: G \to \Aut(\M)$ be an amenable action and $G$
  an unimodular group, for any $1 \leq p \leq \infty$, we have that
  \begin{eqnarray}
    &      & \big\| \Id \rtimes T_m: L_p(\M \rtimes_\theta G) \to L_p(\M \rtimes_\theta G)\big\|_{\cb}\\
    & \leq & \big\| M_m: S_p(L_2 G) \to S_p(L_2 G) \big\|_\cb \nonumber\\
    & \leq & \big\| T_m: L_p(\L G) \to L_p(\L G) \big\|_\cb \nonumber
  \end{eqnarray}
  Furthermore, if $S \in \CB(L_p(\L G))$ is $\theta$-equivariant and $\theta$
  has a $C$-approximating sequence, then
  \begin{eqnarray}
    &       & \big\| S \rtimes \Id: L_p(\M \rtimes_\theta G) \to L_p(\M \rtimes_\theta G)\big\|_\cb\\
    &  \leq & C^\frac{1}{p} \, \big\| S : L_p(\M) \to L_p(\M)\big\|_\cb. \nonumber
  \end{eqnarray}
\end{corollary}

It is still not known whether the amenability condition for the action is
necessary or not for the transference results here presented. Recent
results in the context of transference between Schur and Fourier
multipliers, see \cite{CasSall2015}, and between groups and subgroups,
see \cite{CasParPerrRic2014, GonSalle2016} suggest that amenability
may indeed be necessary. We conjecture the following.

\begin{conjecture}
  \label{CP.cnj.Amen}
  If $\Gamma$ is a discrete group, $\theta: G \to \Aut(\M)$ is a
  trace-preserving action and for some $p \neq 2$
  there is an complete isometry
  \[
    L_p(\M \rtimes_\theta \Gamma) \xrightarrow{\quad j_p \quad} \prod_{\alpha, \U} L_p(\M \weaktensor \B(\ell_2 \Gamma)),
  \]
  satisfying that
  \[
    j_p \, (\Id \rtimes T_m) = (\Id \otimes M_m)^{\alpha, \U} \, j_p,
  \]
  then, the action $\theta$ is amenable.
\end{conjecture}

Observe that, a priori, it is still not known whether the conjecture
above implies that the equality 
\begin{equation}
  \label{CP.eqNorms}
  \| \Id \rtimes T_m:
     L_p(\M \rtimes_\theta \Gamma) \to L_p(\M \rtimes_\theta \Gamma) \|_\cb
  = \| T_m \|_{\cb}
\end{equation}
holds only for amenable actions. It is also unknown if there are multipliers
on $L_p(\L G)$ for which $\|T_m\|_{\cb}$ and $\|\Id \rtimes T_m\|_\cb$ are
not equal.

\section{\bf Stability of maximal bounds}
\label{CP.S3}

Let $\psi: G \to \RR_+$ be a symmetric and conditionally negative function,
see \cite[Appendix C]{BeHarVal2008} for the precise definition. We
have that the functions $e^{-t \, \psi}$ are of positive type and that
they induce a semigroup $S_t = T_{e^{- \psi}}:\L G \to \L G$ of self
adjoint, trace preserving and completely positive maps, such semigroups
are generally referred to as (symmetric) \emph{Markovian semigroups}.
The reader is advised to think of $(S_t)_{t \geq 0}$ as a noncommutative
generalization of the heat semigroup. Such semigroups have attracted
much attention in the abelian setting for their possibilities to
generalize Harmonic analysis to more abstract contexts, see
\cite{Ste1970, VaSaCou1992, Saloff2002}. Spectral
multipliers are the operators defined by functional calculus over the
infinitesimal generator of the semigroup. In our setting such spectral
multipliers are given by Fourier multipliers of the form $T_{m \circ \psi}$,
where $m:\RR_+ \to \CC$. In \cite{GonJunPar2015},
a noncommutative H{\"o}rmander-Mikhlin multiplier
theorem theorem was proved generalizing earlier works of Alexopoulos
\cite{Alex2001Disc,Alex1994Lie}, Hebisch \cite{Heb1992} and
Duong-Ouhabaz-Sikora \cite{DuOuSi2002}. Such result states that,
under certain hypothesis, 
\[
  \| T_{m \circ \psi}: L_p(\L G) \to L_p(\L G) \|_{\cb}
  \lesssim_{(p)}
  \sup_{t \geq 0} \| \eta(\cdot) \, m(t \cdot) \|_{H^{s,\infty}(\RR_+)}, \quad \mbox{ for } \quad 1 < p < \infty,
\]
where $H^{s,\infty}(\RR_+)$ is a Sobolev space with smoothness
order $s$, for $s$ large enough, and $\eta(z)$ is an analytic
function decaying fast at $0$ and $\infty$, see \cite{GonJunPar2015}
for the details. In order to prove such result we introduced an
element $X$ in the extended positive cone of $\L G$, $(\L G)_+^\wedge$,
see \cite{Haa1979, Haa1979ii} for the precise definition, as the
noncommutative analogue of an invariant metric. We regard the triple
$(\L G, \psi, X)$ as a noncommutative generalization of invariant
metric spaces over the dual of the group $G$. The reason behind such
generalization is that we can understood $X$ as the unbounded function
$\chi \mapsto d(e, \chi)$, where $d: \widehat{G} \times \widehat{G} \to \RR_+$
is an invariant metric, recall that by invariance such function determines
the whole metric $d$. To prove the H\"ormander-Mikhlin theorem above we
have to impose certain conditions in $(\L G, \psi, X)$ which we called
the standard assumptions. We recall the definition bellow.

\begin{definition}{\bf (\cite[Definition 2.6]{GonJunPar2015})}
  We say that $(\L G, \psi, X)$ satisfy the, so called,
    \emph{standard assumptions} iff 
  \begin{enumerate}[label={\rm \roman*)}, ref={\rm \roman*)}]
    \item \label{CP.S3.defStdAss.doub}
    The function $\Phi_X(t) = \tau(\chi_{[0,r]}(X))$ is \emph{doubling}, i.e.
    \begin{equation*}
      \Phi_X(2 \, t) \leq C \, \Phi_X(t), \quad \forall \, 0 \leq t.
    \end{equation*}
    \item \label{CP.S3.defStdAss.ii}
    $X$ satisfies the \emph{completely bounded Hardy-Littelwood inequality},
    or, $\mathrm{CBHL}$ in short, for every $1 < p < \infty$, i.e.
    \begin{equation}
      \label{CP.S3.defStdAss.CBHL}
      \tag{{\rm CBHL}}
      \big\| (\R_t)_{t \geq 0} : L_p(\L G) \to L_p(\L G; L_\infty(\RR_+)) \big\|_{\cb} \lesssim_{(p)} 1,  
    \end{equation}
    where $\R_t(x) = \Phi_X(t)^{-1} \, (\chi_{[0,t]}(X) \star x)$ and $\star$ is
    the noncommutative analogue of the convolution over $L_1(\L G)$, given by
    $\lambda(f) \star \lambda(g) = \lambda(f \, g)$.
    \item 
    The Markovian semigroup $S_t = T_{e^{- t \, \psi}}$ has $L_2$-Gaussian bounds
    bounds, i.e.
    \begin{equation*}
      \label{CP.S3.defStdAss.L2GB}
      \tag{$L_2$GB}
      \tau \big\{| \chi_{[r,\infty)}(X) \lambda(e^{-t \, \psi})|^2 \big\}
      \lesssim
      \frac{1}{\Phi_X(\sqrt{t})^\frac{1}{2}} \, e^{- \beta \frac{r^2}{t}}
    \end{equation*}
  \end{enumerate}
\end{definition}

Observe that, following our analogy of $T_{e^{-t \, \psi}}$ with the
heat semigroup, $\lambda(e^{-t \, \psi})$ plays the role of the heat kernel
and \ref{CP.S3.defStdAss.L2GB} is just a form of Gaussian bounds. Similarly, if
$X$ is a classical metric $\chi_{[0,r]}(X) \Phi^{-1}(r)$ is just the
characteristic function of the ball of radius $r$ over the unit after
being normalized in $L_1$ and the maximal norm of \ref{CP.S3.defStdAss.ii} is
just the $L_p$-norm of the Hardy-Littlewood maximal operator.

The goal of this section is to prove that the assumptions defined above
are stable under certain cross-products. Let $(H,\psi_H,X_H)$ and $(G,\psi_G,X_G)$
be triples satisfying the standard assumption and let $\theta:G \to \Aut(H)$
be a $\mu_H$-preserving action. Then, $K = H \rtimes_\theta G$ is again an
unimodular group and it is trivial to check that its Haar measure $\mu_{K}$ can
be identified with $\mu_H \otimes \mu_G$. The action $\theta$ extends to a normal
and $\tau_H$-preserving action on $\L H$. Let $\phi : H \to \CC$ be a function
inducing a normal c.b. multiplier $T_\phi$ over $\L H$. $T_\phi$ is
$\theta$-equivariant, i.e: $T_\phi(\theta_g(x)) = \theta_g(T_\phi(x))$, iff
$\phi$ is $\theta$-invariant, i.e.: $\phi(\theta_g(h)) = \phi(h)$. Therefore,
if $\phi_1 : H \to \CC$ and $\phi_2 : G \to \CC$ are functions of positive
type, the function $\phi:K \to \CC$ given by
\[
  \phi(h,g) = \phi_1(h) \, \phi_2(g)
\]
is also of positive type when $\phi_1$ is $\theta$-invariant.
Indeed, let $\{(h_i, g_i)\}_{i = 1}^{n} \subset K$, then
\begin{equation}
  \label{thetaInvariant}
  \begin{array}{>{\displaystyle}r>{\displaystyle}l>{\displaystyle}lll}
    \Big[ \phi \big( (h_i, g_i)^{-1} \, (h_j, g_j) \big) \Big]_{i,j} & = & \Big[ \phi \big( \theta_{g_i^{-1}} (h_i^{-1} \, h_j), g_i^{-1} \, g_j \big) \Big]_{i,j} &      & \\
                                                                     & = & \big[ \phi_1 ( h_i^{-1} \, h_j ) \, \phi_2 (g_i^{-1} \, g_j) \big]_{i,j}                & \geq & 0.
  \end{array}
\end{equation}
The positivity of the matrix in the last line follows from the fact that 
the Schur product respects positivity. Taking $\phi_1 = e^{-t \psi_H}$
and $\phi_2 = e^{-t \psi_G}$ gives that
$\psi: K \to \RR_+$ given by $\psi(h, g) = \psi_H(h) + \psi_G(g)$ is a
conditionally negative length when $\psi_H$ is $\theta$-invariant.
The next logical step in order to extend the standard assumptions to 
crossed products is to find a way of defining operators 
$X_1 \rtimes \1, \1 \rtimes X_2 \in \L K_+^{\wedge}$ given 
$X_1 \in \L H_{+}^{\wedge}$ and $X_2 \in \L G_+^{\wedge}$. Notice that if
$\pi: \N \to \R$ is any normal $\ast$-homomorphism between von Neumann
algebras, then $\pi$ extends to a normal order-preserving map
$\pi: \N_+^\wedge \to \R_+^\wedge$. Therefore, it makes sense to apply the
$\ast$-automorphisms $\theta_g$ to $X_H$. We will say that $X_H$ is $\theta$
invariant if $\theta_g(X_H) = X_H$ for every $g \in G$. Again, extending the
normal inclusions $\iota_1: \M \into \M \rtimes_\theta G$ and
$\iota_2:\L G \to \M \rtimes_\theta G$ to the extended positive cone gives
operators
\[
  \begin{array}{rclll}
    X_H^2 \rtimes \1 & := & \iota_1( X_H^2 ) & \in & \L K_+^\wedge\\
    \1 \rtimes X_G^2 & := & \iota_2( X_G^2 ) & \in & \L K_+^\wedge.
  \end{array}
\]
and we define the metric $X \in \L K_+^\wedge$ by the following equation
\[
  X^2 = X_H^2 \rtimes \1 + \1 \rtimes X_G^2.
\]

\begin{theorem}
  \label{CP.thm.Stability}
  Let $(H,\psi_H,X_H)$ and $(G,\psi_G,X_G)$ be triples satisfying the standard
  assumptions and $\theta:G \to \Aut(H)$ be a continuous, $\mu_H$-preserving
  action. Assume that $\psi_H$ and $X_H$ are $\theta$-invariant. Then,
  $(K,\psi,X)$, defined as above, is also standard.
\end{theorem}

In the theorem above it is trivial to prove the $L_2$-Gaussian bounds and
doublingness of $\Phi_X$. The key part are the completely bounded
Hardy-Littlewood inequalities. In order to prove that, we are going to
use an $\ell_\infty$-valued version of Theorem
\ref{CP.thm.ultraProdIntertwining}. Notice that we are not imposing
amenability of the action $\theta$ because the standard assumptions force
$G$ to be amenable, see \cite[Remark 2.5]{GonJunPar2015}, and hence any action is
amenable. The stability result for maximal operators will be the following.

\begin{theorem}
  \label{CP.thm.StrongMax}
  Let $\M$ be a hyperfinite von Neumann algebra, $G$ a LCH unimodular 
  group and $\theta:G \to \Aut(\M)$ a $\tau_\M$-preserving action admitting a
  $C$-approximating sequence. Let $(\Omega_1,\nu_1)$ and $(\Omega_2,\nu_2)$
  be measure spaces, $(T_\omega)_{\omega \in \Omega_1}$
  be a family of completely positive Fourier multipliers and
  $(S_\omega)_{\omega \in \Omega_2}$ is a family of completely positive and
  $\theta$-invariant operators satisfying that
  \begin{equation}
    \label{CP.thm.StrongMax.eq1}
    \begin{array}{rclll}
      A & = & \big\| (T_\omega): L_p(\L G) \to L_p(\L G; L_\infty(\Omega_1)) \big\|_\cb & < & \infty\\[2pt]
      B & = & \big\| (S_\omega): L_p(\M) \to L_p(\M; L_\infty(\Omega_2))\big\|_\cb      & < & \infty.
    \end{array}
  \end{equation}
  Then, we have that
  \begin{eqnarray*}
    &      & \big\| (S_\omega \rtimes T_\zeta)_{(\omega,\zeta)}: L_p(\M \rtimes_\theta G) \to L_p(\M \rtimes_\theta G; L_\infty(\Omega_1) \mintensor L_\infty(\Omega_2) ) \big\|_\cb\\
    & \leq & C^\frac{1}{p} \, A \, B.
  \end{eqnarray*}
\end{theorem}

Observe that, in the abelian case with trivial action $\theta = \1$,
Theorem \ref{CP.thm.StrongMax} follows by routinely applying
Fubini-type arguments. We obtain the following corollary.

\begin{corollary}
  \label{CP.cor.DiagMaximal}
   Let $\M$, $G$, $\theta$, $(T_\omega)_{\omega \in \Omega}$ and
   $(S_\omega)_{\omega \in \Omega}$ be like in the previous theorem for
   some fixed $(\Omega,\nu)$. We have that
   \begin{equation*}
     \big\| (S_\omega \rtimes T_\omega)_{\omega}: L_p(\M \rtimes_\theta G) \to L_p(\M \rtimes_\theta G; L_\infty(\Omega)) \big\|_\cb \, \leq \, C^\frac{1}{p} \, A \, B,
   \end{equation*}
   where $A$ and $B$ are defined like in
   \eqref{CP.thm.StrongMax.eq1}.
\end{corollary}

Recall that, since each $T_\omega$ above is a
Fourier multiplier, there is an essentially unique symbol $m_\omega$ such
that $T_\omega = T_{m_\omega}$. In order to prove the theorems above we
need some preliminary results. We will use the following characterization
of boundedness for $L_\infty$-valued Schur multipliers whose proof we omit.

\begin{proposition}
  \label{CP.prp.SchurMuestra}
  Let $(T_{m_\omega})_{\omega \in \Omega} \subset \CB(L_1(\L G))$, we have that
  $(M_{m_\omega})_\omega : S_p(L_2 G) \to S_p[L_\infty(\Omega)]$
  boundedly iff for every $a \in S_p^k$ and 
  $(b^\omega)_\omega \in S_{p'}^k[L_1(\Omega)]$ we have that
  \begin{equation}
    \label{SchurMuestra}
    \Big| \int_\Omega \sum_{i,j}^{k} m_{\omega}(g_i^{-1} g_j) \, a_{i j} \, b_{i j}^{\omega} \, d\,\mu(\omega) \Big|
    \leq
    K \, \| a \|_{S_p^k} \, \| (b^\omega)_\omega \|_{S_{p'}^k[L_1]}.
  \end{equation}
  Furthermore, the optimal $K$ satisfies that
  \[
    K = \big\| (M_{m_\omega})_\omega : S_p \to S_p[ L_\infty(\Omega)] \big\|.
  \]
  The analogous results for complete norms follows after taking
  $a_{i \, j} \in S_p^m$ and $b^{\omega}_{i j} \in S_{p}^{m}$ in
  \eqref{SchurMuestra}
\end{proposition}

The following theorem is just a vector-valued extension of Theorem
\ref{CP.thm.ultraProdIntertwining}.

\begin{theorem}
  \label{ultraProdIntertwining}
  Let $\M$ be a hyperfinite von Neumann algebra and $G$, $\theta$ be as
  above with $\theta$ amenable. For any $1 \leq p < \infty$ and any operator
  space $E$ we have a complete isometry
  \begin{eqnarray*}
    L_p(\M \rtimes_{\theta} G; E) & \xrightarrow{\quad j_p \quad} & \prod_{\alpha, \mathcal{U}} L_p(\M \weaktensor \B(L_2 G); E).
  \end{eqnarray*}
  Furthermore, when $E$ is an operator system $j_p$ is completely positive.
  
  If $E = C(X_i)$ is any abelian $C^\ast$-algebra, where $X_i$ are compact
  Hausdorff spaces, and $(T_{m_x})_{x \in X_2}$ is a family of
  Fourier multipliers in $L_p(\L G)$, then the following diagram
  commute
  \begin{equation*}
    \label{CP.Vector.CrossDiagramII}
    \xymatrix{
      L_p(\M \rtimes_{\theta} G; C(X_1)) \ar[d]^-{(\Id \rtimes T_{m_x})_{x \in X_2}} \ar[rr]^-{j_p} & & \displaystyle{ \prod_{\alpha, \U} L_p(\M \weaktensor \B(L_2 G);C(X_1))} \ar[d]^-{(\Id \otimes M_{m_x})_{x \in X_2}}\\
      L_p(\M \rtimes_\theta G; C(X_1 \times X_2)) \ar[rr]^-{j_p}                                    & & \displaystyle{ \prod_{\alpha, \U} L_p(\M \weaktensor \B(L_2 G);C(X_1 \times X_2))},
    }
  \end{equation*}
  where $M_{m_x}$ is the Herz-Schur multiplier associated with the symbol
  $m_x$.
  Furthermore, if $\theta$ has a $C$-approximating sequence
  and $(S_x)_{x \in X_2}$ are $\theta$-equivariant
  operators in $L_p(\M)$, then 
  \begin{equation*}
    \label{CP.Vector.CrossDiagramI}
    \xymatrix{
      L_p(\M \rtimes_{\theta} G; C(X_1)) \ar[d]^-{(S_{x} \rtimes \Id)_{x \in X_2}} \ar[rr]^-{j_p} & & \displaystyle{ \prod_{\alpha, \U} L_p(\M \weaktensor \B(L_2 G); C(X_1))} \ar[d]^-{(S_{x}^\alpha)^{\alpha, \U}_{x \in X_2}}\\
      L_p(\M \rtimes_\theta G;C(X_1 \times X_2)) \ar[rr]^-{j_p}                                   & & \displaystyle{ \prod_{\alpha, \U} L_p(\M \weaktensor \B(L_2 G); C(X_1 \times X_2))},
    }
  \end{equation*}
  where $(S^\alpha_x)_{x \in X_2}$ satisfies that
  \begin{eqnarray*}
    &      & \big\| (S^\alpha_x)_{x \in X_2}: L_p(\M \weaktensor \B(L_2 G); C(X_1)) \to L_p(\M \weaktensor \B(L_2 G); C(X_1 \times X_2)) \big\|_\cb\\
    & \leq & C^\frac{1}{p} \, \big\| (S_x)_{x \in X_2} : L_p(\M;C(X_1)) \to L_p(\M; C(X_1 \times X_2)) \big\|_\cb.
  \end{eqnarray*}
\end{theorem}

Before going into the proof we would like to clarify why we choose
$C(X)$-valued operators instead of $L_\infty(\Omega)$-valued, for
some measure space $\Omega$, if all we care about are maximal bounds.
The reason is that, in order to pass from the strong maximal type
arguments in Theorem \ref{CP.thm.StrongMax} to the Corollary
\ref{CP.cor.DiagMaximal} we need to restrict the maximal operator
$(S_\omega \rtimes T_\zeta)_{(\omega,\zeta)}$ to the diagonal
$\omega = \zeta$. This requires that the diagonal
restriction operator
$m : L_\infty(\Omega) \otimes L_\infty(\Omega) \to L_\infty(\Omega)$,
given by $m(f \otimes g) = f \, g$, to be completely bounded. That
is not the case is we take
$L_\infty(\Omega) \weaktensor L_\infty(\Omega) = L_\infty(\Omega)$.
Nevertheless it holds if we take
$L_\infty(\Omega) \mintensor L_\infty(\Omega)$, which is not a von
Neumann algebra.

In order to prove Theorem \ref{ultraProdIntertwining} we will need the
following well-known lemma, whose proof we omit.

\begin{lemma}[{\bf \cite{Pi1998}}]
  \label{CP.lem.cpVectorValued}
  Let $\M_1$, $\M_2$ be hyperfinite von Neumann algebras and $E$ an
  operator space. If $\psi: L_p(\M_1) \to L_p(\M_2)$ is a completely
  bounded map, then $\psi \otimes \Id_E: L_p(\M_1;E) \to L_p(\M_2;E)$ is
  completely bounded. Furthermore, if $E$ is an operator system, the map
  $\psi \mapsto \psi \otimes_E$ preserves complete positive maps.
\end{lemma}

\begin{remark}
  \label{CP.rmk.HFreguralMaps}
  When $\M_1 = \M_2=\M$ is hyperfinite and $p=1$, every map $\phi$ satisfying
  that $\phi \otimes \Id_E$ is bounded for every $E$ is actually completely
  bounded, the same follows for
  $p = \infty$ when $\psi$ is normal. For general $p$, the maps $\psi$
  satisfying that $\| \psi \otimes \Id_E: L_p(\M;E) \to L_p(\M;E) \| < \infty$,
  when $E = \ell_\infty$, are called \emph{regular maps} and were
  studied in \cite{Pi1995Regular}. Such maps are exactly those which can
  be expressed as linear combinations of completely positive ones. In the
  non-hyperfinite case the theorem above is false. Indeed, in
  \cite{Haa1985Injectivity}, Haagerup proved that all the completely
  bounded maps in $\M$ are linear combinations of completely
  positive maps iff $\M$ is hyperfinite.
\end{remark}

\begin{proof}{\bf \, (of Theorem \ref{ultraProdIntertwining})}
  Let $(x_\alpha)_\alpha$ be an approximating sequence for the amenable
  action $\theta$. We can construct $X_\alpha$ as in the proof of Theorem
  \ref{CP.thm.ultraProdIntertwining}. For each $j^\alpha_p$ by
  \[
    j_p^\alpha = (\Ad_{X_\alpha^{1/p}} \, j) \otimes \Id_E
  \]
  and by Lemma \ref{CP.lem.cpVectorValued} such maps are complete
  contractions, i.e.
  \[
    \big\| j_p^\alpha: L_p(\M \rtimes_\theta G; E) \to L_p(\M \weaktensor \B(L_2 G); E) \big\|_\cb \leq 1.
  \]
  They are also completely positive when $E$ is an operator system. Let us
  denote temporarily such maps by $j_{p,E}^\alpha$. Now it is enough to
  prove that
  \begin{equation}
    \label{CP.prf.AsympIsoVectorValued}
    \lim_{\alpha, \U} \big\langle (j_{p, \B(\H)}^\alpha \, x), (j_{p',S_1(\H)}^\alpha \, y) \big\rangle = \langle x, y \rangle,
  \end{equation}
  where $\langle \cdot, \cdot \rangle$ is the duality pairing between
  $L_p(\M \rtimes_\theta G; S_1(\H))$ and
  $L_{p'}(\M \rtimes_\theta G; \B(\H))$. That case suffices since we
  can always embed $E$ in a completely isometric way inside $\B(H)$.
  The proof of \eqref{CP.prf.AsympIsoVectorValued} follows like in the
  scalar case. The identity implies that $j_p = (j_p^\alpha)^\U_\alpha$ is
  isometric since we can use that
  \[
    \begin{array}{>{\displaystyle}rc>{\displaystyle}ll}
      \prod_{\alpha, \U} L_{p}(\M;\B(\H)) & \subset & \bigg( \prod_{\alpha, \U} L_{p'}(\M;S_1(\H)) \bigg)^\ast               & \mbox{ when } 1 < p \leq \infty\\
      \prod_{\alpha, \U} L_{1}(\M;\B(\H)) & \subset & \bigg( \prod_{\alpha, \U} L_\infty(\M;S_1(\H)^{\ast \ast}) \bigg)^\ast & \mbox{ otherwise }
    \end{array}
  \]
  and proceed like in the proof of Theorem \ref{CP.thm.ultraProdIntertwining}.
  The commutation identities follow similarly. 
\end{proof}

Theorems \ref{ultraProdIntertwining} gives a way of transferring bounds from
$\Id \otimes M$ where $M$ is a $C(X)$-valued Schur multiplier to
$\Id \rtimes T$, where $T$ is its associated $C(K)$-valued Fourier multiplier.
In order to bound the maximal operator given by Schur multipliers 
$(\Id_\M \otimes M_{m_\omega})_{\omega \in \Omega}$ we need the following 
transference result generalizing \cite[Theorem 4.2]{CasSall2015} to the 
$L_\infty$-valued case.

\begin{theorem}
  \label{transSchurFou}
  Let $G$ be a LCH and unimodular group, $\Omega$ a measure space and 
  $(T_{m_\omega})_{\omega \in \Omega} \subset \CB(L_1(\L G))$ a family of
  Fourier multipliers. If $(M_{m_\omega})_{\omega \in \Omega}$ is the 
  associated family of Schur multipliers then, for every
  $1 \leq p \leq \infty$
  \begin{eqnarray*}
    &      & \big\| (M_{m_\omega})_{\omega \in \Omega}: S_p(L_2 G) \to S_p(L_2 G; L_\infty(\Omega)) \big\|_\cb\\
    & \leq & \big\| (T_{m_\omega})_{\omega \in \Omega}: L_p(\L G), \, L_p(\L G; L_\infty(\Omega)) \big\|_\cb.
  \end{eqnarray*}
\end{theorem}

\begin{proof}
  Let $\mu$ be a probability measure over $\Omega$ such that 
  $L_1(\Omega, \mu)^\ast = L_\infty(\Omega)$, by \cite[Lemma 4.1]{CasSall2015}
  there is a dense subset of exponents $1 \leq p \leq \infty$ such that we can
  choose sequences $x_n$ and $y_n$ of norm one elements in $L_p(\L G)$ and 
  $L_{p'}(\L G)$ such that
  \[
    \lim_{n} \, \langle y_n, T_m x_n\rangle = m(e).
  \]
  Let us define $z_n = y_n \otimes \chi_{\Omega} \in L_p(\L G;L_1(\Omega, \mu))$.
  Since the $L_1(\Omega; L_p(\L G))$-norm is larger then the
  $L_p(\L G;L_1(\Omega))$-norm we obtain that $\|z_n\|_{L_p(\L G;L_1)} \leq 1$.
  Furthermore, if $(T_{m_\omega})_{\omega \in \Omega}$ is a family of
  multiplier as in the hypothesis, then
  \begin{equation}
    \label{integralApp}
    \lim_{n} \, \langle z_n, T_{m_\omega} x_n \rangle = \int_\Omega m_{\omega}(e) \, d\,\mu(\omega),
  \end{equation}
  where the parying $\langle \cdot, \cdot \rangle$ is the duality
  pairing between $L_p(\L G;L_\infty)$ and $L_{p'}(\L G;L_1)$. Proving
  \begin{eqnarray}
      &       & \Big| \int_\Omega \sum_{i,j}^{k} m_{\omega}(g_i^{-1} g_j) \, a_{i j} \, b_{i j}^{\omega} \, d\,\mu(\omega) \Big| \nonumber\\
      & \leq  & \big\|(T_\omega)_\omega: L_p(\L G) \to L_p(\L G; L_\infty(\Omega)) \big\|_\cb \, \| a \|_{S^k_p} \, \|(b^{\omega})_\omega\|_{S_p^k[L_1]} \label{SchurDuality}
  \end{eqnarray}
  implies, by Proposition \ref{SchurMuestra}, that
  \[
    \|(M_{m_\omega})_\omega\|_{\B(S_p, \, S_p(L_\infty))}
    \leq
    \big\|(T_{m_\omega})_\omega:L_p(\L G) \to L_p(\L G; L_\infty) \big\|_\cb \, \| a \|_{S^k_p} \|(b^\omega)_{\omega}\|_{S_p^k[L_1]}.
  \]
  In order to obtain the same bound for the complete norms it is enough to
  take $a_{i j} \in S_p^m$ and repeat the calculation. Therefore to prove
  the claim it suffices to prove (\ref{SchurDuality}). Pick $x_n$ and $z_n$
  like in (\ref{integralApp}) and consider $A_n \in S_p^k[L_p(\L G)]$ and
  $B_n^\omega \in S_{p'}^k[L_{p'}(\L G; L_1(\Omega))]$ given by
  \begin{eqnarray*}
    A_n        & = & u^* \, (a \otimes x_n) \, u \\
    B^\omega_n & = & u^* \, (b^\omega \otimes z_n) \, u
  \end{eqnarray*}
  where $u$ is the unitary in $M_k \otimes \L G$ given by
  \[
    u = 
      \begin{pmatrix}
        \lambda_{g_1} & 0             & \cdots & 0     \\
        0             & \lambda_{g_2} & \cdots & 0     \\
        \vdots        &               &        & \vdots \\
        0             & 0             & \cdots & \lambda_{g_k}    
      \end{pmatrix}
  \]
  As a consequence $\|A_n\|_{S_p^k[L_p(\L G)]} = \|x_n\|_{L_p(\L G)}$ and
  $\|B_n\|_{S_{p'}^k[L_{p'}(\L G;L_1)]} \leq \|z_n\|_{L_{p'}(\L G; L_1)}$. Observe
  that $\lambda_{g_i} \, T_m(\lambda_{g_i}^* x \lambda_{g_j}) \lambda_{g_j}^*
  = T_{m_{i j}}(x)$, where $m_{i j}(h) = m(g_{i}^{-1} \, h \, g_{j})$,
  therefore
  \begin{eqnarray*}
    &      & \int_\Omega \sum_{i,j}^{k} m_{\omega}(g_i^{-1} g_j) a_{i j} b_{i j}^{\omega} \, d\,\mu(\omega) \nonumber\\
    & =    & \int_\Omega \sum_{i,j}^{k} a_{i j} b_{i j}^{\omega} \lim_n \langle z_n, T_{m_{i j}^\omega} x_n \rangle \, d\,\mu(\omega)\\
    & =    & \lim_n \int_\Omega \sum_{i,j}^{k} a_{i j} b_{i j}^{\omega} \langle z_n, T_{m_{i j}^\omega} x_n \rangle \, d\,\mu(\omega)\\
    & =    & \lim_n \langle (B^{\omega}_n)_\omega, ((\Id \otimes T_{m_\omega}) A_n)_\omega \rangle\\
    & \leq & \big\| (T_{m_\omega})_\omega:L_p(\L G) \to L_p(\L G; L_\infty) \big\|_\cb \, \|a\|_{S_p^k} \|(b^\omega)_\omega\|_{S_p^k[L_1]}.
  \end{eqnarray*}
  We have used the Dominated Convergence Theorem to exchange the limit and
  the integral in the second equation, which is justified since the multipliers
  $m_\omega$ are uniformly bounded.
\end{proof}

We can pass to the proof of the strong maximal bounds. Since we are going to
reduce the problem to that of tensor product it is convenient to recall the
following modification of the result for tensor products, see
\cite[Lemma 2.8]{GonJunPar2015}, whose proof is a trivial consequence of
\eqref{CP.S0.max}. We include the proof for the sake of completeness.

\begin{lemma}
  \label{CP.StrongTensor}
  Let $(M_i,\tau_i)$, for $i \in \{1,2\}$ be two hyperfinite von Neumann
  algebras with n.s.f. traces, $(\Omega_i, \nu_i)$ two measure spaces and
  $(S_\omega)_{\omega \in \Omega_1}$, $(T_\omega)_{\omega \in \Omega_2}$
  be families of completely positive operators satisfying that
  \[
    {\setstretch{1.25}
      \begin{array}{rcl>{\displaystyle}ll}
        A_1 & := & \big\| (T_\omega)_{\omega \in \Omega_1}: L_p(\M_1) \to L_p(\M_1;L_\infty(\Omega_1)) \big\|_\cb & < & \infty\\
        A_2 & := & \big\| (S_\omega)_{\omega \in \Omega_2}: L_p(\M_2) \to L_p(\M_2;L_\infty(\Omega_2)) \big\|_\cb & < & \infty
      \end{array}.
    }
  \]
  Then, we have that
  \[
    \big\| (R_{\omega, \zeta})_{(\omega,\zeta)}:
    L_p(\M_1 \weaktensor \M_2) \to L_p(\M_1 \weaktensor \M_2; L_\infty(\Omega_1) \mintensor L_\infty(\Omega_2)) \big\|_\cb
    \leq A_1 \, A_2,
  \]
  where $R_{\omega, \zeta} = \Ad_Y \, (T_\omega \otimes \Id) \, \Ad_X \, (\Id \otimes S_\zeta)$
  and $X, Y \in \M_1 \weaktensor \M_2$ are self adjoint contractive operators.
\end{lemma}

\begin{proof}
  Recall that if $\phi:E \to F$ is completely bounded, then
  $\| \Id \otimes \phi: S_p^n[E] \to S_p^n[F] \| \leq \|\phi\|_\cb$. As a
  consequence, the same is true for
  $\Id \otimes \phi: L_p(\M;E) \to L_p(\M;F)$, when $\M$ is hyperfinite. We are 
  going to use also that $L_p(\M_1;L_p(\M_2;L_\infty)) =
  L_p(\M_1 \weaktensor \M_2;L_\infty)$. By complete positivity of
  $(\Id \otimes S_\zeta)$ and \eqref{CP.S0.max} we have that for every
  $x \in L_p(\M_1 \weaktensor \M_2)$ there is another
  $u$ with $\| u \|_p \leq (1 + \epsilon) \, B \| x \|_p$, where
  $\epsilon$ can be taken arbitrarily small. Now
  \begin{eqnarray*}
    R_{\omega, \zeta}(x) & =    & \Ad_Y \, (T_\omega \otimes \Id) \, \Ad_X \, (\Id \otimes S_\zeta)(x)\\
                         & \leq & \Ad_Y \, (T_\omega \otimes \Id) \, \Ad_X \, u
  \end{eqnarray*}
  and applying the same procedure to $\Ad_X \, u$ once again gives
  the desired identity.
\end{proof}

\begin{proof}{\bf \, (of Theorem \ref{CP.thm.StrongMax}).}
  Recall that for any measure space $\Omega$, the algebra $L_\infty(\Omega)$
  is isomorphic to $C(X)$ where $X$ is certain compact Hausdorff and
  disconnected topological space. In order to apply Theorem
  \ref{ultraProdIntertwining} we need to express an element
  $\omega \mapsto T_\omega$ inside $L_\infty(\Omega;\CB(L_p(\N)))$ as a
  $\CB(L_p(\N))$-valued function on $C(X)$. But, since
  $X \subset \Ball(L_\infty(\Omega)^\ast)$, we can safely evaluate
  $\phi \otimes \Id$, where $\phi \in X$, against $(T_\omega)_\omega$. By an
  application of Theorem \ref{ultraProdIntertwining} the diagram in
  Figure \ref{diagramaMaximal} commutes. 
  \begin{figure}[t]
    \centering
    \xymatrix{
      L_p(\M \rtimes_\theta G) \ar[ddr]^-{(S_\omega \rtimes \Id )_{t}}  \ar[rr]^-{(S_\omega \rtimes T_\zeta)_{(\omega,\zeta)}} \ar[ddddd]^-{j_p} & & L_p(\M \rtimes_\theta G; L_\infty  \mintensor L_\infty) \ar[ddddd]^-{j_p}\\
      \\
      & L_p(\M \rtimes_\theta G; L_\infty) \ar[uur]^{(\Id \rtimes T_\zeta)_{\zeta}} \ar[d]^-{j_p} &\\
      & \displaystyle{\prod_{n,\U}} L_p(\M \weaktensor \B(L_2 G); L_\infty) \ar[ddr]^-{(\Id \otimes T_\zeta)_\zeta} &\\
      \\
      \displaystyle{\prod_{n,\U}} L_p(\M \weaktensor \B(L_2 G)) \ar[uur]^-{(S_\omega^\alpha)^{\alpha, \U}_\omega} & & \displaystyle{\prod_{n,\U}} L_p(\M \weaktensor \B(L_2 G); L_\infty  \mintensor L_\infty)
    }
    \caption{
      Commutative diagram for the proof of Theorem
      \ref{CP.thm.StrongMax}.
    }
    \label{diagramaMaximal}
    \hrulefill
  \end{figure}
  The $j_p$ are the complete isometries of Theorem
  \ref{ultraProdIntertwining}. The isometries
  $j_p$ intertwine $(S_\omega \rtimes \Id )$ with the ultraproduct with
  respect to $\U$ in $\alpha$ of the maps
  \[
    S^\alpha_\omega = \Ad_{X_\alpha^{1/p}} \, (S_\omega \otimes \Id) \, \Ad_{Y_\alpha^{1/p}}
  \]
  and so $(S_\omega \rtimes \Id )_{\omega \in \Omega_2}$ is completely bounded
  (resp. completely positive) if the ultraproduct of such maps 
  is completely bounded (resp. completely positive). But, since each
  $S_\omega$ is c.p. and $\M$ is hyperfinite that follows by Lemma
  \ref{CP.lem.cpVectorValued}. Similarly,
  $(\Id \rtimes T_\zeta)_{\zeta \in \Omega_2}$ is completely bounded (resp.
  completely positive) if $(\Id \otimes M_\zeta)_{\zeta \in \Omega_2}$ is c.b.
  (resp. c.p.), where $M_\zeta$ is the Schur multiplier sharing its symbol
  with $T_\zeta$. By Theorem \ref{transSchurFou} $(\Id \otimes M_\zeta)_{\zeta}$
  is completely bounded. Now, applying Lemma \ref{CP.StrongTensor} gives
  that $((\Id \otimes M_\zeta) \, S^\alpha_\omega)_{(\omega,\zeta)}$
  is completely bounded and that finishes the proof.
\end{proof}

The Corollary \ref{CP.cor.DiagMaximal} follows from the Theorem above after
applying Lemma \ref{CP.lem.diagRestriction}.

\begin{lemma}
  \label{CP.lem.diagRestriction}
  Let $\A$ be an abelian $C^\ast$-algebra and $(\M,\tau)$ a hyperfinite
  von Neumann algebra with a n.s.f. trace $\tau$, then
  \begin{equation*}
    \big\| (\Id_\M \otimes m): L_p(\M;\A \mintensor \A) \to L_p(\M;\A) \big\|_\cb \leq 1,
  \end{equation*}
  where $m: \A \mintensor \A \to \A$ is given by $f \otimes g \mapsto f \, g$.
\end{lemma}

\begin{proof}{\bf \, (of Corollary \ref{CP.cor.DiagMaximal})}
  Notice that, if $\R_1 = (S_\omega \rtimes T_\zeta)_{(\omega,\zeta)}$ and
  $\R_2 = (S_\omega \rtimes T_\omega)_{\omega}$, we have that:
  \[
    \R_2 = (\Id_{\M \rtimes G} \otimes m) \, \R_1,
  \]
  and applying Lemma \ref{CP.lem.diagRestriction} together with Theorem
  \ref{CP.thm.StrongMax} gives the desired result.
\end{proof}

With that at hand we can pass to prove of the stability under crossed products
of the standard assumptions.

\begin{proof}{\bf \, (of Theorem \ref{CP.thm.Stability})}
  To prove that $X$ is doubling we just use that $X^2_H \rtimes \1$ and 
  $\1 \rtimes X^2_G$ commute when $X_H$ is $\theta$ invariant and therefore:
  \[
    \chi_{[0,r^2)}(X) \leq \chi_{[0,r^2)}(X_H \rtimes \1) \, \chi_{[0,r^2)}(\1 \rtimes X_G).
  \]
  Using that
  $\tau_K((x \rtimes \1) \, (\1 \rtimes y)) = \tau_H(x) \, \tau_G(y)$ gives
  that $\Phi_X(r) \leq \Phi_{X_H}(r) \Phi_{X_G}(r)$. Similarly it can be proved 
  that $\Phi_{X_H}(r) \Phi_{X_G}(r) \leq \Phi_X(2 \, r)$ and therefore $X$ is 
  doubling. The \ref{CP.S3.defStdAss.L2GB} property is proved similarly. For the
  \ref{CP.S3.defStdAss.CBHL} maximal inequalities we just use that
  \[
    \frac{\chi_{[0,r]}(X)}{\Phi_X(r)} \star u
    \lesssim_{(D_{\Phi_{X_H}}, D_{\Phi_{X_G}})}
     (\R_r^H \rtimes \R_r^G)(u),
  \]
  where
  \[
    \begin{array}{rc>{\displaystyle}lcrc>{\displaystyle}l}
      \R_r^H(u) & = & \frac{\chi_{[0,r]}(X_H)}{\Phi_{X_H}(r)} \star u & \mbox{ and } & \R_r^G(u) & = & \frac{\chi_{[0,r]}(X_G)}{\Phi_{X_G}(r)} \star u.
    \end{array}
  \]
  The maximal boundedness of $(\R_r^H \rtimes \R_r^G)_{r \geq 0}$ follows from
  Corollary \ref{CP.cor.DiagMaximal}.
\end{proof}

\textbf{Acknowledgement.} The author was informed, through a personal
communication, that some of the results here exposed were obtained, in
the context of certain amenable groups instead of amenable actions,
independently by Quanhua Xu. I am also thankful to {\'E}ric Ricard for
some discussions concerning the necessity of the amenability condition
for the action.

\bibliographystyle{alpha}
\bibliography{../bibliography/bibliography}

\begin{thebibliography}{BdlHV08}

\bibitem[AEG94]{AdamEllGior1994}
Scot Adams, George~A. Elliott, and Thierry Giordano.
\newblock Amenable actions of groups.
\newblock {\em Trans. Amer. Math. Soc.}, 344(2):803--822, 1994.

\bibitem[Ale94]{Alex1994Lie}
G.~Alexopoulos.
\newblock Spectral multipliers on {L}ie groups of polynomial growth.
\newblock {\em Proc. Amer. Math. Soc.}, 120(3):973--979, 1994.

\bibitem[Ale01]{Alex2001Disc}
G.~Alexopoulos.
\newblock Spectral multipliers on discrete groups.
\newblock {\em Bulletin of the London Mathematical Society}, 33:417--424, 7
  2001.

\bibitem[BdlHV08]{BeHarVal2008}
B.~Bekka, P.~de~la Harpe, and A.~Valette.
\newblock {\em Kazhdan's Property ({T})}.
\newblock New Mathematical Monographs. Cambridge University Press, 2008.

\bibitem[BF91]{BoFend1991}
Marek Bo{\.{z}}ejko and Gero Fendler.
\newblock Herz-schur multipliers and uniformly bounded representations of
  discrete groups.
\newblock {\em Archiv der Mathematik}, 57(3):290--298, 1991.

\bibitem[BO08]{BroO2008}
Nathanial~P. Brown and Narutaka Ozawa.
\newblock {\em {$C^*$}-algebras and finite-dimensional approximations},
  volume~88 of {\em Graduate Studies in Mathematics}.
\newblock American Mathematical Society, Providence, RI, 2008.

\bibitem[CdlS15]{CasSall2015}
M.~Caspers and M.~de~la Salle.
\newblock Schur and {F}ourier multipliers of an amenable group acting on
  non-commutative {$L_p$}-spaces.
\newblock {\em Trans. Amer. Math. Soc.}, 2015.

\bibitem[CPPR15]{CasParPerrRic2014}
Martijn Caspers, Javier Parcet, Mathilde Perrin, and {\'E}ric Ricard.
\newblock Noncommutative de leeuw theorems.
\newblock In {\em Forum of Mathematics, Sigma}, volume~3, page e21. Cambridge
  Univ Press, 2015.

\bibitem[DOS02]{DuOuSi2002}
T.X. Duong, E.M. Ouhabaz, and A.~Sikora.
\newblock Plancherel-type estimates and sharp spectral multipliers.
\newblock {\em J. Funct. Anal.}, 196(2):443--485, 2002.

\bibitem[ER00]{EffRu2000Book}
E.~G. Effros and Z.-J. Ruan.
\newblock {\em Operator {S}paces}, volume~23 of {\em London Mathematical
  Society Monographs. New Series}.
\newblock The Clarendon Press, Oxford University Press, New York, 2000.

\bibitem[Fur73]{Furs1973Boundary}
Harry Furstenberg.
\newblock Boundary theory and stochastic processes on homogeneous spaces.
\newblock In {\em Harmonic analysis on homogeneous spaces ({P}roc. {S}ympos.
  {P}ure {M}ath., {V}ol. {XXVI}, {W}illiams {C}oll., {W}illiamstown, {M}ass.,
  1972)}, pages 193--229. Amer. Math. Soc., Providence, R.I., 1973.

\bibitem[GPJP15]{GonJunPar2015}
A.~M. {G}onz\'alez {P}\'erez, M.~Junge, and J.~Parcet.
\newblock Smooth {F}ourier multipliers in group von {N}eumann algebras via
  {S}obolev dimension.
\newblock {\em to appear in {A}nn. {S}ci. de l'{\'E}cole {N}ormale
  {S}up{\'e}rieure}, 2015.

\bibitem[GPS16]{GonSalle2016}
A.~M. Gonz{\'a}lez-{P}{\'e}rez and {M}ikael de~la {S}alle.
\newblock Optimal conditions for multiplier restriction results over
  noncommutative {$L_p$}-spaces.
\newblock unpublished, 2016.

\bibitem[Haa79a]{Haa1979}
U.~Haagerup.
\newblock Operator-valued weights in von {N}eumann algebras. {I}.
\newblock {\em J. Funct. Anal.}, 32(2):175--206, 1979.

\bibitem[Haa79b]{Haa1979ii}
U.~Haagerup.
\newblock Operator-valued weights in von {N}eumann algebras. {II}.
\newblock {\em J. Funct. Anal.}, 33(3):339--361, 1979.

\bibitem[Haa85]{Haa1985Injectivity}
Uffe Haagerup.
\newblock Injectivity and decomposition of completely bounded maps.
\newblock In {\em Operator algebras and their connections with topology and
  ergodic theory ({B}u\c steni, 1983)}, volume 1132 of {\em Lecture Notes in
  Math.}, pages 170--222. Springer, Berlin, 1985.

\bibitem[Har99]{Har1999Fourier}
A.~Harcharras.
\newblock Fourier analysis, {S}chur multipliers on {$S^p$} and non-commutative
  {$\Lambda(p)$}-sets.
\newblock {\em Studia Math.}, 137(3):203--260, 1999.

\bibitem[Heb92]{Heb1992}
Waldemar Hebisch.
\newblock On heat kernels on {L}ie groups.
\newblock {\em Math. Z.}, 210(4):593--605, 1992.

\bibitem[HR00]{HigRoe2000Novikov}
Nigel Higson and John Roe.
\newblock Amenable group actions and the {N}ovikov conjecture.
\newblock {\em J. Reine Angew. Math.}, 519:143--153, 2000.

\bibitem[JMP14]{JunMeiPar2014}
M.~Junge, T.~Mei, and J.~Parcet.
\newblock Smooth {F}ourier multipliers on group von {N}eumann algebras.
\newblock {\em Geom. Funct. Anal.}, 24(6):1913--1980, 2014.

\bibitem[JMP15]{JunMeiPar2014NoncommRiesz}
M.~Junge, T.~Mei, and J.~Parcet.
\newblock Noncommutative {R}iesz transforms---{D}imension free bounds and
  {F}ourier multipliers.
\newblock {A}rXiv:1407.2475, 2015.

\bibitem[Jun02]{Jun2002Doob}
M.~Junge.
\newblock Doob's inequality for non-commutative martingales.
\newblock {\em J. Reine Angew. Math.}, 549:149--190, 2002.

\bibitem[JX07]{JunXu2007}
M.~Junge and Q.~Xu.
\newblock Noncommutative maximal ergodic theorems.
\newblock {\em J. Amer. Math. Soc.}, 20(2):385--439, 2007.

\bibitem[KW99]{KirchWass1999}
Eberhard Kirchberg and Simon Wassermann.
\newblock Exact groups and continuous bundles of {$C^*$}-algebras.
\newblock {\em Math. Ann.}, 315(2):169--203, 1999.

\bibitem[LdlS11]{LaffSall2011}
V.~Lafforgue and M.~de~la Salle.
\newblock Noncommutative {$L^{p}$}-spaces without the completely bounded
  approximation property.
\newblock {\em Duke Math. J.}, 160(1):71--116, 10 2011.

\bibitem[NR11]{NeuRic2011}
S.~Neuwirth and E.~Ricard.
\newblock {Transfer of Fourier multipliers into Schur multipliers and sumsets
  in discrete group.}
\newblock {\em Can. J. Math.}, 63(5):1161--1187, 2011.

\bibitem[Oza00]{Oz2000}
N.~Ozawa.
\newblock Amenable actions and exactness for discrete groups.
\newblock {\em C. R. Acad. Sci. Paris S\'er. I Math.}, 330(8):691--695, 2000.

\bibitem[Oza06]{Oz2006}
Narutaka Ozawa.
\newblock Amenable actions and applications.
\newblock In {\em International {C}ongress of {M}athematicians. {V}ol. {II}},
  pages 1563--1580. Eur. Math. Soc., Z\"urich, 2006.

\bibitem[Pat88]{Pa1988Amenability}
Alan L.~T. Paterson.
\newblock {\em Amenability}, volume~29 of {\em Mathematical Surveys and
  Monographs}.
\newblock American Mathematical Society, Providence, RI, 1988.

\bibitem[Ped79]{Ped1979}
G.~K. Pedersen.
\newblock {\em {$C^{\ast} $}-algebras and their automorphism groups}, volume~14
  of {\em London Mathematical Society Monographs}.
\newblock Academic Press, Inc. [Harcourt Brace Jovanovich, Publishers],
  London-New York, 1979.

\bibitem[Pis95a]{Pi1995}
Gilles Pisier.
\newblock Multipliers and lacunary sets in non-amenable groups.
\newblock {\em Amer. J. Math.}, 117(2):337--376, 1995.

\bibitem[Pis95b]{Pi1995Regular}
Gilles Pisier.
\newblock Regular operators between non-commutative {$L_p$}-spaces.
\newblock {\em Bull. Sci. Math.}, 119(2):95--118, 1995.

\bibitem[Pis98]{Pi1998}
G.~Pisier.
\newblock Non-commutative vector valued {$L_p$}-spaces and completely
  {$p$}-summing maps.
\newblock {\em Ast\'erisque}, (247), 1998.

\bibitem[Pis03]{Pi2003}
G.~Pisier.
\newblock {\em Introduction to {O}perator {S}pace {T}heory}, volume 294 of {\em
  London Mathematical Society Lecture Note Series}.
\newblock Cambridge University Press, Cambridge, 2003.

\bibitem[PX03]{PiXu2003}
G.~Pisier and Q.~Xu.
\newblock Non-commutative {$L^p$}-spaces.
\newblock In {\em Handbook of the geometry of {B}anach spaces, {V}ol.\ 2},
  pages 1459--1517. North-Holland, Amsterdam, 2003.

\bibitem[SC02]{Saloff2002}
Laurent Saloff-Coste.
\newblock {\em Aspects of {S}obolev-type inequalities}, volume 289 of {\em
  London Mathematical Society Lecture Note Series}.
\newblock Cambridge University Press, Cambridge, 2002.

\bibitem[Ste70]{Ste1970}
E.M. Stein.
\newblock {\em Topics in harmonic analysis related to the {L}ittlewood-{P}aley
  theory.}
\newblock Annals of Mathematics Studies, No. 63. Princeton University Press,
  Princeton, N.J., 1970.

\bibitem[Var85]{Va1985QuantumBook}
V.~S. Varadarajan.
\newblock {\em Geometry of quantum theory}.
\newblock Springer-Verlag, New York, second edition, 1985.

\bibitem[VSCC92]{VaSaCou1992}
N.~Th. Varopoulos, L.~Saloff-Coste, and T.~Coulhon.
\newblock {\em Analysis and geometry on groups}, volume 100 of {\em Cambridge
  Tracts in Mathematics}.
\newblock Cambridge University Press, Cambridge, 1992.

\bibitem[Zim77]{Zimmer1977}
Robert~J. Zimmer.
\newblock Hyperfinite factors and amenable ergodic actions.
\newblock {\em Invent. Math.}, 41(1):23--31, 1977.

\bibitem[Zim78a]{Zimmer1978}
Robert~J. Zimmer.
\newblock Amenable ergodic group actions and an application to {P}oisson
  boundaries of random walks.
\newblock {\em J. Functional Analysis}, 27(3):350--372, 1978.

\bibitem[Zim78b]{Zimmer1978Pairs}
Robert~J. Zimmer.
\newblock Amenable pairs of groups and ergodic actions and the associated von
  {N}eumann algebras.
\newblock {\em Trans. Amer. Math. Soc.}, 243:271--286, 1978.

\bibitem[Zim78c]{Zimmer1978Induced}
Robert~J. Zimmer.
\newblock Induced and amenable ergodic actions of {L}ie groups.
\newblock {\em Ann. Sci. \'Ecole Norm. Sup. (4)}, 11(3):407--428, 1978.

\bibitem[Zim84]{Zimmer1984}
R.~J. Zimmer.
\newblock {\em Ergodic theory and semisimple groups}, volume~81 of {\em
  Monographs in Mathematics}.
\newblock Birkh\"auser Verlag, Basel, 1984.

\end{thebibliography}

\end{document}